\documentclass[11pt]{article}
\usepackage{amsfonts, amsmath, amssymb}
\usepackage{mathrsfs}
\usepackage{amsthm}
\usepackage{hyperref}
\usepackage{enumitem}
\usepackage{color}

\newtheorem{prop}{Proposition}
\newtheorem{defi}{Definition}
\newtheorem{coro}{Corollary}
\newtheorem{rem}{Remark}

\newtheorem{theorem}{Theorem}

\def\LL{\left}
\def\RR{\right}

\def\PP{\mathbb{P}}
\def\dom{{\rm Dom}}
\def\Ph{\phantom{\int}\!\!\!\!\!\!\!\!}
\def\ph{\phantom{1^l\!\!\!\!\!\!\!}}

\def\EE{\mathbb{E}}
\def\norm#1{[\hskip-.5pt]#1[\hskip-.5pt]}

\begin{document}

\title{Large time behavior of semilinear stochastic partial differential
equations perturbed by a mixture of Brownian and fractional Brownian motions
}
\author{Marco Dozzi\footnote{corresponding author, marco.dozzi@univ-lorraine.fr, UMR-CNRS 7502, Institut Elie Cartan de Lorraine, Nancy, France} \and Ekaterina T.
Kolkovska\thanks{%
Centro de Investigaci\'{o}n en Matem\'{a}ticas, Guanajuato, Mexico.} \and %
Jos\'{e} A. L\'{o}pez-Mimbela$^{\dag}$ \and Rim Touibi\thanks{UMR-CNRS 7502, Institut Elie Cartan de Lorraine, Nancy, France.}}\date{ }
\maketitle

\begin{abstract}
We study the trajectorywise blowup behavior of a semilinear partial
differential equation that is driven by a mixture of multiplicative Brownian
and fractional Brownian motion, modeling different types of random
perturbations. The linear operator is supposed to have an eigenfunction of
constant sign, and we show its influence, as well as the influence of its
eigenvalue and of the other parameters of the equation, on the occurrence of
a blowup in finite time of the solution. We give estimates for the
probability of finite time blowup and of blowup before a given fixed time.
Essential tools are the mild and weak form of an associated random partial
differential equation.


\end{abstract}

\section{Introduction}

In this paper we study existence, uniqueness and the blowup behavior of
solutions to the fractional stochastic partial differential equation of the
form%
\begin{eqnarray}  \label{2.1}
du(x,t) &=&\left[\frac{1}{2}k^{2}(t)Lu(x,t)+g(u(x,t))\right]%
dt+u(x,t)\,dN_{t},\quad x\in D, \quad t>0,  \notag \\
u(x,0) &=&\varphi (x)\geq 0,  \notag \\
u(x,t) &=&0, \quad x\in \partial D, \quad t \geq 0,
\end{eqnarray}%
where $D\subset
\mathbb{R}
^{d}$ is a bounded Lipschitz domain, $L$ is the infinitesimal generator of a
strongly continuous semigroup {\color{black} of contractions} which satisfies conditions (\ref{P}), (\ref{L}) below,  and $\varphi \in L^{\infty}(D),$ where $L^{\infty
}(D)$ is the space of real-valued essentially bounded functions on $D.$
Additionally, $g$ is a nonnegative locally Lipschitz function and $N$ is a
process given by
\begin{equation}\label{Def-of-N}
N_t= \int_0^t a(s) \,dB(s)+ \int_0^t b(s) \,dB^H(s), \quad t\geq0,
\end{equation}
where $B$ is Brownian motion and $B^{H}$ is fractional Brownian motion with
Hurst parameter $H>1/2$, $a$ is continuous and $b$ is H\"{o}lder continuous
of order $\alpha >1-H.$ Both, $B$ and $B^{H},$ are supposed to be defined on
a filtered probability space $(\Omega ,\mathcal{F},(\mathcal{F}_{t},t\geqq
0),\mathbb{P})$ and adapted to the filtration $(\mathcal{F}_{t},t\geqq 0).$
Such models have recently been studied under the name of `mixed models' in
the context of stochastic differential equations, see \cite{MS} and \cite%
{MRS}. When $N=0,$ $L=\Delta ,$ $k=1,$ $g(u)=u^{1+\beta }$ we obtain the
classical Fujita equation which was studied in \cite{Fuj}. In \cite{DL} and
\cite{ALP} there were considered the cases when $N$ is a Brownian motion, in
\cite{DKL} it was investigated the case when $N$ is a fractional Brownian
motion with Hurst parameter $H>1/2$ and $D\subset\mathbb{R}^{d}$, and in
\cite{D-K-LM-SAA} the case of $H\ge 1/2$ and $D=\mathbb{R}^{d}$.

The fractional Brownian motion (fBm) appears in many stochastic phenomena,
where rough external forces are present. The principal difference, compared
to Brownian motion, is that fBm is not a semimartingale nor a Markov
process, hence classical theory of stochastic integration cannot be applied. Since $H>1/2$, the stochastic integral with respect to $%
B^{H}$ in \eqref{2.1} can be understood as a fractional integral.  Also the presence of both, Brownian and fractional Brownian motion in \eqref{2.1}, due to their
different analytic and probabilistic properties, modelize different aspects
of the random evolution in time of the solution. The factor
$k^2/2$ in front of $L$ affects dissipativity, which in several cases is in
favor of retarding or even preventing blowup.

We consider both, weak and mild solutions of \eqref{2.1}, which we prove are
equivalent and unique. Beyond existence and uniqueness of weak and mild
solutions we are interested in their qualitative behaviour. In Theorem \ref%
{THM1} below we obtain a random time $\tau^*$ which is an upper bound of the
explosion time $\tau.$
 In Theorem \ref{THM7} we obtain a lower bound $%
\tau_* $ of $\tau$ so that a.s.
\begin{equation*}
\tau_*\le \tau\le \tau^*.
\end{equation*}
The random times $\tau_*$ and $\tau^*$ are given by exponential functionals
of the mixture of a Brownian and a fractional Brownian motion. The laws of
such kind of functionals presently are not known.
 In order to study the distribution of $\tau^*$
we use  the well-known  representation  of $B^H$
 in the form $$B_{t}^{H}=\int_{0}^{t}{K^{H}(t,s)\,dW_{s}}, $$ where
 the kernel $K^H$ is given in \eqref{Kernel-K} and
 $W$ is a  Brownian motion defined in the same filtered probability space as $B$.
In general, $W$  can be different from the Brownian motion $B $ appearing in the first integral of \eqref{Def-of-N}. We obtain
estimates of the probability $\mathbb{P}(\tau<\infty) $, and of the tail
distribution of $\tau^*$.
To achieve this we make use of recent results of N.T. Dung \cite{D,DII} from
the Malliavin theory for continuous isonormal Gaussian processes.

  In Theorem \ref{THM2}  we obtain upper bounds for $\PP(\tau^*\le T)$ in the case when $B=W,$ and in Theorem \ref{THM3} when $B$ is independent of $W$, and when $B$ and $W$ are general Brownian motions. In  Theorem \ref{THM4} we obtain lower bounds for $\PP(\tau< \infty) $ when $B=W$.
   As a result in the case when $W=B$ we get specific configurations of the
coefficients $a, b$ and $k$ under which the weak solution (hence also the
mild solution) of equation \eqref{2.1} exhibits finite time blow-up. To
be concrete suppose that $g(z) \geq Cz^{1+\beta}$ for some constants $C>0$, $%
\beta>0$, $B_{t}^{H}=\int_{0}^{t}{K^{H}(t,s)\,dB_{s}}, $
and
\begin{equation*}
\int_0^t a^2(r)\,dr \sim t^{2l}, \quad \int_0^tb^2(r)\,dr\sim t^{2m}, \quad \int_0^t k^2(r)\,dr \sim t^{2p}
\quad%
\mbox{ as }\quad t \to \infty
\end{equation*}
for some nonnegative constants $l,\, m$ and $p$. If {\color{black} $\beta \in (0,1/2)$ and $\max\{p,l\}>
H+m-1/2$, or if $\beta=1/2$ and $p>H+m-1/2$, or if $\beta > 1/2$ } and $p>\max\{l, H+m-1/2\},$ then
all nontrivial positive solutions of (\ref{2.1}) suffer finite-time blowup
with positive probability.

Our approach here is to transform the equation \eqref{2.1} into a random
partial differential equation (RPDE) (\ref{2.2}), whose solution blows up at
the same random time $\tau$ as the solution of \eqref{2.1}, and to work with
this equation. The blowup behavior of (\ref{2.2}) is easier to determine
because $N$ appears as a coefficient, and not as stochastic integrator as in
(\ref{2.1}). Such transformations are indeed known for more general SPDEs
than (\ref{2.1}), including equations whose stochastic term does not depend
linearly on $u$, see \cite{LR}. But for the RPDE's associated to more
general SPDE's it seems difficult to find explicit expressions for upper and
lower bounds for the blowup time, and this is an essential point in our
study. Another reason for having chosen the relatively simple form of %
\eqref{2.1} and \eqref{2.2} is that we consider the blowup trajectorywise
which is a relatively strong notion compared, e.g., to blowup of the moments
of the solution (see, e.g. \cite{Cho}). The crucial ingredient in the proofs
is the existence of a positive eigenvalue and an eigenfunction with constant
sign of the adjoint operator of $L.$ Special attention is given to the case $%
H\in (\frac{3}{4},1)$ because then the process $N$ is equivalent to a
Brownian motion \cite{Ch}. This allows us to apply a result by Dufresne and
Yor \cite{Y} on the law of exponential functionals of the Brownian motion to
get in Theorem \ref{THM5} an explicit lower bound for the probability of
blowup in finite time. \bigskip


We finish this section by introducing some notations and definitions we will
need in the sequel. A stopping time $\tau :\Omega \rightarrow (0,\infty )$
with respect to the filtration $(\mathcal{F}_{t},t\geqq 0)$ is a blowup time
of a solution $u$ of\ \eqref{2.1} if
\begin{equation*}
\limsup_{t\nearrow \tau }\sup_{x\in D} |u(x,t)| =+\infty \quad \mbox{$\mathbb{P}$-a.s.}
\end{equation*}

\noindent 
{\color{black}
Let $(P_{t}^{D},t\geqq 0)$ and $((P^{D})_{t}^{\ast
},t\geqq 0)$
} be the strongly continuous semigroups corresponding to the
operator $L$ and its adjoint $L^*:$
\begin{equation}\label{D}
\int_{D}f(x)P_{t}^{D}g(x)dx=\int_{D}g(x)(P^{D})_{t}^{\ast }f(x)dx,\quad
f,g\in {\color{black}L^2(D).}
\end{equation}
{\color{black}
As usual, $Lf:=\underset{t\rightarrow 0}{\lim }\frac{1}{t}(P_{t}^{D}f-f)$
for all $f\in {\color{black} L^2(D)}$ in the domain of $L,$ denoted by $\mathrm{Dom}(L).$
Due to the Hille-Yosida theorem, $\mathrm{Dom}(L)$ and $\mathrm{Dom}(L^{\ast })$ are dense
in $L^2(D).$
}
Let $P_{t}^{D}(x,\Gamma )$ and ($P^{D})_{t}^{\ast }(x,\Gamma )$ denote the
associated transition functions, where $t> 0,$ $x\in D,$ and $\Gamma \in
\mathcal{B}(D),$ the Borel sets on $D$. In the sequel we will assume that they admit
densities, i.e. there exist families of continuous functions $(p^{D}(t,\cdot
,\cdot ),t>0)$ and $((p^{D})^{\ast }(t,\cdot ,\cdot ),t>0)$ on $D\times D$
such that%
\begin{eqnarray*}
P_{t}^{D}g(x) &=&\int_{D}g(y)P_{t}^{D}(x,dy)=\int_{D}g(y)p^{D}(t,x,y)dy, \\
(P^{D})_{t}^{\ast }f(x) &=&\int_{D}f(y)(P^{D})_{t}^{\ast
}(x,dy)=\int_{D}f(y)(p^{D})^{\ast }(t,x,y)dy.
\end{eqnarray*}
Due to \eqref{D},
\begin{equation}\label{PDt1}
(p^{D})^{\ast }(t,x,y)=p^{D}(t,y,x)\quad\mbox{for all $t > 0$ and $%
x,y\in D.$}
\end{equation}

\section{The weak solution of the associated random partial differential
equation, equivalence with the mild solution}

\label{Section2}

Let us consider the random partial differential equation
\begin{eqnarray}  \label{2.2}
\frac{\partial v}{\partial t}(x,t)&=&\frac{1}{2}k^{2}(t)Lv(x,t)-\frac{1}{2}%
a^{2}(t)v(x,t)+\exp (-N_{t})g(\exp (N_{t})v(x,t)), \\
v(x,0) &=&\varphi (x),\:x\in D,  \notag \\
v(x,t)&=&0, \ t \geq 0,\ x\in \partial{D}.   \notag
\end{eqnarray}
In this section we transform the weak form of \eqref{2.1} into the weak
form of \eqref{2.2} 
using the transformation  $v(x,t)=\exp (-N_{t})u(x,t),\, x\in D,\, t\ge 0. $ Hence, if blowup takes place in finite time, it occurs of course at the same
time and at the same place $x\in D$ for the solutions of both equations.

In the following we write $\left\langle \cdot ,\cdot \right\rangle _{D}$ for
the scalar product in $L^{2}(D).$

\begin{defi}
An $(\mathcal{F}_{t},t\geqq 0)$-adapted random field $v=(v(x,t),$ $t\in
\lbrack 0,T],\ x\in D)$ with values in $L^{2}(D)$ is a weak solution of %
\eqref{2.2} if, for all $t\in $ $[0,T]$ and all $f\in \mathrm{Dom}(L^{\ast })
$, $\mathbb{P}$-a.s.
\begin{eqnarray*}
\left\langle v(\cdot ,t),f\right\rangle _{D} &=&\left\langle \varphi
,f\right\rangle _{D}+\int_{0}^{t}\left( \frac{1}{2}k^{2}(s)\left\langle
v(\cdot ,s),L^{\ast }f\right\rangle _{D}-\frac{1}{2}a^{2}(s)\left\langle
v(\cdot ,s),f\right\rangle _{D}\right) ds
\end{eqnarray*}%
\begin{equation}
+\int_{0}^{t}\exp (-N_{s})\left\langle g(\exp (N_{s})v(\cdot
,s)),f\right\rangle _{D}ds.  \label{weak rpde}
\end{equation}
\end{defi}

Since $g$ is supposed to be locally Lipschitz, a blowup in finite time of $v$
may occur, and the blowup time $\tau $ depends in general on $\omega \in
\Omega .$ A \emph{weak solution of \eqref{2.2} up to $\tau $} is defined as
an $(\mathcal{F}_{t},t\geqq 0)$-adapted random field $v$ that satisfies %
\eqref{weak rpde} for all $t\in (0,T \wedge \tau )$ $\mathbb{P}$-a.s. If $%
\omega $ is such that $v(\omega ,\cdot ,\cdot )$ does not blowup in finite
time, we set $\tau (\omega )=\infty .$

\begin{defi}
An $(\mathcal{F}_{t},t\geqq 0)$-adapted random field $u=(u(x,t),$ $t\in
\lbrack 0,T],x\in D)$ with values in $L^{2}(D)$ is a weak solution of %
\eqref{2.1} up to $\tau $ if, for all $t\in $ $(0,T\wedge \tau )$ and all $%
f\in \mathrm{Dom}(L^{\ast }),$ $\mathbb{P}$-a.s.
\begin{equation*}
(i)\text{ }\int_{0}^{t}a^{2}(s)\left[1+\left\langle u(\cdot
,s),f\right\rangle _{D}^{2}\right]\,ds<\infty ,\quad b(\bullet )\left\langle
u(\cdot ,\bullet ),f\right\rangle _{D}\in \mathcal{C}^{\beta }[0,t] \mbox{
for some $\beta >1-H,$}
\end{equation*}
\begin{equation*}
(ii)\text{ \ }\int_{0}^{t}\left( k^{2}(s)\left\vert \left\langle u(\cdot
,s),L^{\ast }f\right\rangle _{D}\right\vert +\left\vert \left\langle
g(u(\cdot ,s)),f\right\rangle _{D}\right\vert \right)\, ds<\infty , %
\phantom{XXXXXXxXXXXXXX}
\end{equation*}
and%
\begin{equation*}
\left\langle u(\cdot ,t),f\right\rangle _{D}=\left\langle \varphi
,f\right\rangle _{D}+\int_{0}^{t}\left( \frac{1}{2}k^{2}(s)\left\langle
u(\cdot ,s),L^{\ast }f\right\rangle _{D}+\left\langle g(u(\cdot
,s),f\right\rangle _{D}\right) ds
\end{equation*}%
\begin{equation}
+\int_{0}^{t}\left\langle u(\cdot ,s),f\right\rangle _{D}dN_{s}.
\label{weak spde}
\end{equation}
\end{defi}

\noindent Conditions (i) and (ii) in the above definition are sufficient for
the It\^{o}, the fractional and the Lebesgue integrals in \eqref{weak spde}
to be well defined $\mathbb{P}$-a.s.

We proceed now to the relation between \eqref{weak spde} and \eqref{weak
rpde}.

\begin{prop}
\label{PROP1} If $u$ is a weak solution of \eqref{2.1} up to a random time $%
\tau $, then $v(x,t)=\exp (-N_{t})u(x,t)$ is a weak solution of \eqref{2.2}
up to $\tau $, and viceversa.
\end{prop}

\begin{rem}
We notice that $\left\langle v(\cdot ,s),f\right\rangle _{D}$ is
absolutely continuous in $s$ if $v$ is a weak solution of \eqref{2.2}. With
the choice $u(x,t):=\exp(N_{t})v(x,t)$ condition (i) is satisfied. In fact,
for $t<T\wedge \tau (\omega ),$
\begin{equation*}
\int_{0}^{t}\left\langle u(\cdot ,s),f\right\rangle
_{D}a(s)\,dB_{s}=\int_{0}^{t}\left\langle v(\cdot ,s),f\right\rangle
_{D}\exp (N_{s})a(s)\,dB_{s}
\end{equation*}%
is well defined since $\int_{0}^{t}(\int_{D}v(x,s)f(x)\,dx)^{2}\exp
(2N_{s})a^{2}(s)\,ds<\infty$ \ $\mathbb{P}$-a.s.

Recall that the fractional integral $\int_0^Tf(x)dg(x)$ is defined (in the sense of Zähle \cite{Z}) in \cite[Def. 2.1.1]{M} for $f,g$ belonging to fractional Sobolev spaces. If $0<\varepsilon<H$, $f$ and $g$ are Hölder continuous of exponents $\alpha$ and $H-\varepsilon$ respectively, and $\alpha+H-\varepsilon>1$, this fractional
integral
coincides with the corresponding generalized Riemann-Stieltjes integral; see \cite[Thm. 2.1.7]{M}.
Hence, the fractional integral
\begin{equation}
\int_{0}^{t}\left\langle u(\cdot ,s),f\right\rangle
_{D}b(s)\,dB_{s}^{H}=\int_{0}^{t}\left\langle v(\cdot ,s),f\right\rangle
_{D}\exp (N_{s})b(s)\,dB_{s}^{H}  \label{frac int}
\end{equation}
is well defined  for $t<T\wedge \tau (\omega )$
because, on the one hand,  $N_{\cdot }=\int_{0}^{\cdot }(a(s)\,dB_{s}+b(s)\,dB_{s}^{H})$ is $%
\mathbb{P}$-a.s. H\"{o}lder continous of order $1/2-\epsilon $ for all $%
\epsilon >0$ by the theorem of Kolmogorov and \cite[Proposition 4.1]{NR}. On the other hand
 $b(\cdot)$ is $\alpha$-Hölder continuous (with $\alpha>1-H$) and $B^H$ is Hölder continuous with exponent $H-\varepsilon$ for any $\varepsilon>0$. Hence, choosing $\varepsilon <\min\{H/2-1/4, \alpha +H-1\}$ we get that the
 integrand on the right side of \eqref{frac int} is H\"{o}lder
continuous of order $\min\{\alpha,1/2-\varepsilon\}$, and therefore $H-\varepsilon +\min\{\alpha, 1/2 - \varepsilon\}>1$ and the integral is well defined as
a generalized Riemann-Stieltjes integral.

\end{rem}
\begin{proof}
Let $T>0.$ It suffices to prove the assertion for $t\in (0,T\wedge \tau $)$.$
We apply (a slight generalisation of) the It\^{o} formula in \cite[page 184]{M}.
Let
\begin{eqnarray*}
Y_{t}^{1} &=&\int_{0}^{t}a(s)\,dB_{s}\text{ },\text{ \ }Y_{t}^{2}=\int_{0}^{t}%
\left\langle u(\cdot ,s),f\right\rangle _{D}a(s)\,dB_{s}, \\
Y_{t}^{3} &=&\int_{0}^{t}b(s)\,dB_{s}^{H},\text{\ \ }Y_{t}^{4}=\int_{0}^{t}%
\left\langle u(\cdot ,s),f\right\rangle _{D}b(s)\,dB_{s}^{H},\text{ } \\
Y_{t}^{5} &=&\left\langle \varphi ,f\right\rangle _{D}+\int_{0}^{t}\left(
\frac{1}{2}k^{2}(s)\left\langle u(\cdot ,s),L^{\ast }f\right\rangle
_{D}+\left\langle g(u(\cdot ,s)),f\right\rangle _{D}\right) ds,
\end{eqnarray*}%
and let $F(y_{1},y_{2},y_{3},y_{4},y_{5})=\exp
(-y_{1}-y_{3})(y_{5}+y_{2}+y_{4}). $ Then
\[
F(Y_{t}^{1},\ldots,Y_{t}^{5})=\exp (-N_{t})\left\langle u(\cdot
,t),f\right\rangle _{D}=\left\langle v(\cdot ,t),f\right\rangle _{D}.
\]%
The above mentioned It\^{o} formula then reads
\begin{eqnarray*}\lefteqn{
F(Y_{t}^{1},\ldots,Y_{t}^{5})}\\
&=&F(Y_{0}^{1},\ldots,Y_{0}^{5})+\sum_{i=1}^{5}\int_{0}^{t}\frac{\partial F}{%
\partial y_{i}}(Y_{s}^{1},\ldots,Y_{s}^{5})\,dY_{s}^{i}
+\frac{1}{2}\sum_{i,j=1}^{2}\int_{0}^{t}\frac{\partial ^{2}F}{\partial
y_{i}\partial y_{j}}(Y_{s}^{1},\ldots,Y_{s}^{5})\,d\left\langle
Y_{s}^{i},Y_{s}^{j}\right\rangle .
\end{eqnarray*}%
Since $u$ is a weak solution of \eqref{2.1},
\begin{eqnarray*}
\left\langle v(\cdot ,t),f\right\rangle _{D} &=&\left\langle \varphi
,f\right\rangle _{D}-\int_{0}^{t}\exp (-N_{s})\left\langle u(\cdot
,s),f\right\rangle _{D}\LL(a(s)\,dB_{s}+b(s)\,dB_{s}^{H}\RR) \\
&&+\int_{0}^{t}\exp (-N_{s})\left\langle u(\cdot ,s),f\right\rangle
_{D}\LL(a(s)dB_{s}+b(s)dB_{s}^{H}\RR) \\
&&+\int_{0}^{t}\exp (-N_{s})\left( \frac{1}{2}k^{2}(s)\left\langle u(\cdot
,s),L^{\ast }f\right\rangle _{D}+\left\langle g(u(\cdot ,s)),f\right\rangle
_{D}\right) ds \\
&&-\frac{1}{2}\int_{0}^{t}\exp (-N_{s})\left\langle u(\cdot
,s),f\right\rangle _{D}a^{2}(s)ds \\
&=&\left\langle \varphi ,f\right\rangle _{D}+\int_{0}^{t}\left( \frac{1}{2}%
k^{2}(s)\left\langle v(\cdot ,s),L^{\ast }f\right\rangle _{D}-\frac{1}{2}%
a^{2}(s)\left\langle v(\cdot ,s),f\right\rangle _{D}\right) \,ds \\
&&+\int_{0}^{t}\exp (-N_{s})\left\langle g(\exp (N_{s})v(\cdot
,s)),f\right\rangle _{D}\,ds.
\end{eqnarray*}
Therefore $v$ is a weak solution of \eqref{2.2}. Similarly we obtain the viceversa result. \end{proof}

In order to define the mild solutions of equations \eqref{2.1} and  %
\eqref{2.2} we define first the evolution families of contractions
corresponding to the generator $\frac{1}{2}k^{2}(t)L.$ For $0\le s<t$ let
\begin{equation}  \label{AK}
K(t,s)=\frac{1}{2}\int_{s}^{t}k^{2}(r)\,dr,\quad A(t,s)=\frac{1}{2}%
\int_{s}^{t}a^{2}(r)\,dr,\quad K(t)=K(t,0),\quad A(t)=A(t,0),
\end{equation}
and set $p^{D}(s,x;t,y)=p^{D}(K(t,s),x,y),$ $x,y\in D\times D,$ $0\leqq s<t.$
For $f\in L^{2}(D)$ the corresponding evolution families of contractions on $%
L^{2}(D)$ are given by
\begin{eqnarray*}
U^{D}(t,s)f(x) &=&\int_{D}p^{D}(s,x;t,y)f(y)dy=P_{K(t,s)}^{D}f(x), \\
\text{\ }(U^{D})^{\ast }(t,s)f(x)
&=&\int_{D}p^{D}(s,y;t,x)f(y)dy=(P^{D})_{K(t,s)}^{\ast }f(x).\text{ }
\end{eqnarray*}

\begin{defi}
An $(\mathcal{F}_{t},t\geqq 0)$-adapted random field $v=(v(x,t),$ $t\geqq
0,\ x\in D)$ with values in $L^{2}(D)$ is a mild solution of~\eqref{2.2} on $%
[0,T]$ if, for all $t\in $ $[0,T],$ $\mathbb{P}$-a.s.
\begin{eqnarray*}
v(x,t) &=&U^{D}(t,0)\varphi (x)-\frac{1}{2}%
\int_{0}^{t}a^{2}(s)U^{D}(t,s)v(x,s)\,ds \\
&&+\int_{0}^{t}\exp (-N_{s})U^{D}(t,s)\left(g((\exp N_{s})%
\phantom{1^l\!\!\!\!\!\!\!}\:\: v(x,s))\right)\,ds.\text{ }
\end{eqnarray*}
\end{defi}

\begin{prop}
\label{Proposition2} The mild form of~\eqref{2.2} can be written as
\begin{eqnarray}
v(x,t)&=&\exp (-A(t))U^{D}(t,0)\varphi (x)  \notag \\
&+&\int_{0}^{t} \exp (-N_{s}-A(t,s))U^{D}(t,s)g(\exp (N_{s})v(\cdot
,s))(x)ds,\qquad  \label{ms}
\end{eqnarray}
where $A(t,s)$ and $A(t)$ are given in \eqref{AK}.
\end{prop}

\begin{rem}
Since $g$ and $\varphi$ are supposed to be nonnegative,
\begin{equation*}
v(x,t)\ge \exp (-A(t))U^{D}(t,0)\varphi (x)\ge0 \mbox{ for all $x\in D$ and
$t\ge 0.$}
\end{equation*}
\end{rem}

\begin{proof}
Let $w(x,t)=\exp (A(t))v(x,t).$ For $f\in L^2(D),$ we get from
the definition of the mild solution%
\begin{eqnarray*}
\frac{d}{dt}\langle w(\cdot ,t),f\rangle _{D} &=&\frac{1}{2}a^{2}(t)\exp
(A(t))\langle v(\cdot ,t),f\rangle _{D}+\exp (A(t))\frac{d}{dt}\langle
v(\cdot ,t),f\rangle _{D} \\
&=&\frac{1}{2}a^{2}(t)\exp (A(t))\langle v(\cdot ,t),f\rangle _{D} \\
&&+\exp (A(t))\left( \frac{1}{2}k^{2}(t)\left\langle v(\cdot ,t),L^{\ast
}f\right\rangle _{D}-\frac{1}{2}a^{2}(t)\left\langle v(\cdot
,t),f\right\rangle _{D}\right) \\
&&+\exp (A(t))\exp (-N_{t})\left\langle g(\exp (N_{t})v(\cdot
,t)),f\right\rangle _{D} \\
&=&\frac{1}{2}\exp (A(t))k^{2}(t)\left\langle v(\cdot ,t),L^{\ast
}f\right\rangle _{D}+\exp (A(t)-N_{t})\left\langle g(\exp (N_{t})v(\cdot
,t)),f\right\rangle _{D} \\
&=&\frac{1}{2}k^{2}(t)\left\langle w(\cdot ,t),L^{\ast }f\right\rangle
_{D}+\exp (A(t)-N_{t})\left\langle g(\exp (N_{t}-A(t))w(\cdot
,t)),f\right\rangle _{D},
\end{eqnarray*}
with boundary conditions $w(x,0)=\varphi (x)$ for $x\in D$ and $w(x,t)=0$
for $x\in \partial D.$ Therefore $w$ is a weak solution of the RPDE formally
given by%
\begin{equation*}
\frac{d}{dt}w(x,t)=\frac{1}{2}k^{2}(t)Lw(x,t)+\exp (A(t)-N_{t})g(\exp
(N_{t}-A(t))w(x,t)).
\end{equation*}
By the definition of the mild solution%
\begin{equation*}
w(x,t)=U^{D}(t,0)\varphi (x)+\int_{0}^{t}\exp (A(s)-N_{s})U^{D}(t,s)g(\exp
(N_{s}-A(s))w(\cdot ,s))(x)\,ds.
\end{equation*}
Consequently,%
\begin{eqnarray*}
v(x,t) &=&\exp (-A(t))w(x,t) \\
&=&\exp (-A(t))U^{D}(t,0)\varphi (x)+\int_{0}^{t}ds\exp
(-A(t,s)-N_{s})U^{D}(t,s)g(\exp (N_{s})v(\cdot ,s))(x).
\end{eqnarray*}
\end{proof}


{\color{black}
\begin{theorem}
\label{PROP3} The equation \eqref{ms} has a unique non-negative local mild
solution, i.e. there exists $t>0$ such that \eqref{ms} has a mild solution
in $L^{\infty }([0,t[\times D).$
\end{theorem}
\begin{proof}
Let  $T>0$ and denote
$
E_T=\LL\{ v:[0,T]\times D \to L^{\infty }(D) : \norm{v} <\infty \RR \},
$ where $$\norm{v}:=\sup_{0\leq t\leq T}\| v(t, \cdot)\| _{\infty}. $$ Let $P_T=\{    v\in E_T: v\geq0\} $ and for $R>0$ let $C_R=\{v\in E_T : \norm{v} \leq R \}. $
Then $E_T$ is a Banach space and $P_T$ and $C_R$ are closed subsets of $E_T$.  Let us now define
$$\psi(v)(t,x)=
e^{-A(t)}U^{D}(t,0)\varphi (x)+\int_{0}^{t}e^{
-A(t,s)-N_{s}}U^{D}(t,s)g\LL(e^{N_{s}}v(\cdot ,s)\RR)(x)\,ds. $$
We will prove that for sufficiently big $R$  and sufficiently small $T$, $\psi$  is contraction on $P_T \cap C_R$.
Let $v_1, v_2\in  P_T \cap C_R. $ Then
$$
\norm{
\psi(v_1)-\psi(v_2)
}
\
\le
\
\sup_{0\le t\le T}
\int_0^t
\LL\|
 e^{-N_s}\LL[g\LL(e^{N_s}v_1\RR)-g\LL(e^{N_s}v_2\RR)\RR]\RR\|_{\infty}\,ds.
$$
	Let $A_T=\sup_{0\le s\le T}e^{|N_s|}$ and $G_R=\sup_{|x|<R}|g(x)|$, and assume that $g$ is locally Lipschitz with Lipschitz constant $K_R$ 
	in the ball of radius $R>0$ centered at $0$. Then,
	$$
	\sup_{0\le s\le T}\LL\|e^{-N_s}v_i(s,\cdot)\RR\|_{\infty}\ \le \ A_TR,\quad i=1,2,
	$$
and
$$
\LL\|
e^{-N_s}g\LL(
e^{N_s}v_1(s,\cdot)\RR)
-
e^{-N_s}g\LL(
e^{N_s}v_2(s,\cdot)\RR)
\RR\|_{\infty}\ \le \ A^2_T K_{A_TR}\LL\|v_1(s)-v_2(s)\RR\|_{\infty}.
$$
	Therefore,
	$$
	\norm{\psi(v_1)-\psi(v_2)}\ \le \ \sup_{0\le t\le T}\int_0^t
	A_{T}^2K_{A_TR}\,\norm{v_1-v_2}\,ds\ = \ TA^2_TK_{A_tR}\,\norm{v_1-v_2}.
	$$
	We need
	\begin{equation}\label{LL1}
	TA^2_TK_{A_TR}<1.
	\end{equation}
	In addition, we require that $C_R\cap P_T$ be mapped by $\psi$ into itself. Let $v\in C_R\cap P_T$. Using that for $0\le s\le T$ the operator $U^D(t,s)$ is a contraction, and that
$\|e ^{N_s}v(\cdot,s)\|_{\infty}\le A_T R$,  we get
$\|g(e^{N_s}v(\cdot,s))\|_{\infty}\le G_{A_TR}.$ It follows that
	$$
	\norm{\psi(v)} \ \le   \|\varphi\|_{\infty}
	+
	\sup_{0\le t\le T}\int_0^t\LL\|e^{-N(s)}\RR\|_{\infty}\,ds\,G_{A_TR}\
	\le\  \|\varphi\|_{\infty} +TA_TG_{A_TR}.
	$$
	Hence, we need that
	\begin{equation}\label{LL2}
	\|\varphi\|_{\infty}+TA_TG_{A_TR}\ <\ R.
	\end{equation}
	Let $R$ be such that $R\ge 2\|\varphi\|_{\infty}$. Since $\lim_{T\to0}A_T=1$, we choose $\varepsilon_1>0$ so that $A_T<2$ if $T<\varepsilon_1$, and
	$$
	\varepsilon \ < \ \frac{R}{4G_{2R}}\wedge\frac{1}{4K_{2R}}\wedge \varepsilon_1.
	$$
	Using that $G_A\le G_B$ and $ K_A\le K_B$ if $A\le B$, we get for
$R>2\|\varphi\|_{\infty}$ and $T<\varepsilon$,
	$$
\|\varphi\|_{\infty} + TA_TG_{A_TR} \ \le\  \|\varphi\|_{\infty} +2\varepsilon G_{2R} < \frac{R}{2} +\frac{R}{2} =R
	$$
	and
	$$
	TA^2_TK_{A_TR}< 4\varepsilon K_{2R}<1.
	$$
\end{proof}	
}

We proceed to prove equivalence of weak and mild solutions of \eqref{2.2}.
The proof of this theorem follows the method in \cite[Theorem 9.15]{PZ},
where this equivalence is shown for SPDE's with autonomous differential
operators and driven by L\'{e}vy noise. For a comparison of weak and mild
solutions of SPDEs driven by fractional Brownian motion we refer to \cite%
{333}.

We state first the Kolmogorov backward and forward equations for $U^{D}.$ By
the Kolmogorov backward equation for $P^{D}$, the transition density $%
p^{D}(u,x,y)$ satisfies, for any $y$ fixed, $\frac{\partial }{\partial u}%
p^{D}(u,x,y)=Lp^{D}(u,x,y).$ Then $(s,x)\rightarrow $ $p^{D}(s,x;t,y)$
satisfies, for $(t,y)$ fixed, the equation
\begin{equation*}
-\frac{\partial }{\partial s}p^{D}(s,x;t,y)=-\frac{\partial }{\partial s}%
p^{D}(K(t,s),x,y)=-\frac{\partial }{\partial u}p^{D}(u,x,y)\mid _{u=K(t,s)}%
\frac{\partial }{\partial s}K(t,s)
\end{equation*}
\begin{equation}
=\frac{1}{2}k^{2}(s)Lp^{D}(K(t,s),x,y)=\frac{1}{2}k^{2}(s)Lp^{D}(s,x;t,y).
\label{backward}
\end{equation}
Similarly, by the Kolmogorov forward equation for $P^{D},$ for any $x$
fixed, $p^{D}(u,x,y)$ satisfies
$$\frac{\partial }{\partial u}
p^{D}(u,x,y)=L^{\ast }p^{D}(u,x,y).$$  Then $(t,y)\rightarrow $ $%
p^{D}(s,x;t,y) $ satisfies, for $(s,x)$ fixed, the equation%
\begin{equation*}
\frac{\partial }{\partial t}p^{D}(s,x;t,y)=\frac{\partial }{\partial t}%
p^{D}(K(t,s),x,y)=\frac{\partial }{\partial u}p^{D}(u,x,y)\mid _{u=K(t,s)}%
\frac{\partial }{\partial t}K(t,s)
\end{equation*}
\begin{equation}
=\frac{1}{2}k^{2}(t)L^{\ast }p^{D}(K(t,s),x,y)=\frac{1}{2}k^{2}(t)L^{\ast
}p^{D}(s,x;t,y).  \label{forward}
\end{equation}

\begin{theorem}
{\color{black}\label{THM8} Consider the random partial differential equation \eqref{2.2}.
 Then $v$ is a weak solution of \eqref{2.2} on $%
[0,T]$ if and only if $v$ is a mild solution of \eqref{2.2} on $[0,T]$.
}
\end{theorem}


\begin{proof}
 Assume that $v$ is a weak solution of
\eqref{2.2}. Let {\color{black}$h\in {C}^{1}([0,\infty),\mathbb{R})$}, $f\in {\dom(L}^{\ast }),$ and $%
G(x,t):=-\frac{1}{2}a^{2}(t)v(x,t)+\exp (-N_{t})g(\exp (N_{t})v(x,t)).$
The integration by parts formula is applicable since $h\in
{\color{black} {C}^{1}([0,\infty),\mathbb{R})}
$ (see \cite{PZ} Proposition 9.16) and yields%
\begin{eqnarray*}
\langle v(\cdot ,t),h(t)f(\cdot )\rangle _{D}
&=&\langle v(\cdot ,0),h(0)f(\cdot )\rangle _{D} +\int_{0}^{t}\langle v(\cdot
,s),h^{\prime }(s)f(\cdot )\rangle _{D}\, ds\\ \nonumber 
&+&\int_{0}^{t}\langle v(\cdot ,s),\frac{1}{2}h(s)k^{2}(s)L^{\ast }f(\cdot
)\rangle _{D}\, ds+\int_{0}^{t}\langle G(\cdot ,s),h(s)f(\cdot )\rangle _{D} \,ds.
\end{eqnarray*}%
{\color{black}
Since the functions $h\cdot f$ are dense in $C^1([0,\infty),{\rm Dom}(L^*))$, for each $z\in C^1([0,\infty),{\rm Dom}(L^*))$ we have
\begin{eqnarray}\label{5}
\langle v(\cdot ,t),z(\cdot,t )\rangle _{D}
&=&\langle v(\cdot ,0),z(\cdot,0)\rangle _{D} +\int_{0}^{t}\langle v(\cdot
,s),\frac{\partial}{\partial s}z(\cdot,s)\rangle _{D}\, ds\\ \nonumber 
&+&\int_{0}^{t}\langle v(\cdot ,s),\frac{1}{2}k^{2}(s)L^{\ast }z(\cdot,s
)\rangle _{D}\, ds+\int_{0}^{t}\langle G(\cdot ,s),z(\cdot,s )\rangle _{D} \,ds.
\end{eqnarray}
For each $f\in{\rm Dom}(L^*) $ we define
\begin{equation*}
\psi (x,s):=(U^{D})^{\ast }(t,s)f(x)=\left\{
\begin{tabular}{ll}
$\langle p^{D^{\ast }}(s,x;t,\cdot ),f(\cdot )\rangle _{D}$ & if $s<t$, \\
&  \\
$f(x)$ & if $s=t$,%
\end{tabular}%
\right.
\end{equation*}%
hence $\psi\in C^1([0,\infty),{\rm Dom}(L^*))$.
Taking $z=\psi(x,s)$ in \eqref{5} we get, for any $t\in \lbrack 0,T]$ fixed,}
\begin{eqnarray}\nonumber
\langle v(\cdot ,t),\psi (\cdot ,t)\rangle _{D} &=&\langle v(\cdot ,0),\psi
(\cdot ,0)\rangle _{D}+\int_{0}^{t}\left\langle v(\cdot ,s),\frac{d}{ds}\psi
(\cdot ,s)+\frac{1}{2}k^{2}(s)L^{\ast }\psi (\cdot ,s)\right\rangle _{D}ds \\ \label{sept-1}
&&+\int_{0}^{t}\langle G(\cdot ,s),\psi (\cdot ,s)\rangle _{D}\,ds.
\end{eqnarray}%
Now we evaluate the  terms above:%
\begin{eqnarray*}
 \langle v(\cdot ,0),\psi (\cdot ,0)\rangle _{D}
&=&\int_{D}v(x,0)\int_{D}p^{D^{\ast }}(0,x;t,y)f(y)\,dy\,dx \\
&=&\int_{D}f(y)\int_{D}p^{D^{\ast }}(0,x;t,y)v(x,0)\,dx\,dy
 \ = \ \left\langle U^{D}(t,0)v(\cdot ,0),f(\cdot )\right\rangle _{D}.
\end{eqnarray*}
By applying the Kolmogorov backward equation to $(x,s)\rightarrow (U^{D})^{\ast }(t,s)f(x)$ we get
\begin{eqnarray*}
-\frac{d}{ds}\psi (x,s) &=&-\frac{\partial }{\partial s}\left\langle
(p^{D})^{\ast }(s,x;t,\cdot ),f(\cdot )\right\rangle _{D} \\
&=&\frac{1}{2}k^{2}(s)L^{\ast }\left\langle (p^{D})^{\ast }(s,x;t,\cdot
),f(\cdot )\right\rangle _{D}=\frac{1}{2}k^{2}(s)L^{\ast }\psi (x,s).
\end{eqnarray*}
Moreover, from Fubini's theorem and \eqref{PDt1}
\begin{eqnarray*}
\langle G(\cdot ,s),\psi (\cdot ,s)\rangle _{D}
&=&\int_{D}G(x,s)\int_{D}p^{D^{\ast }}(s,x;t,y)f(y)\,dy\,dx \\
&=&\int_{D}f(y)\int_{D}p^{D}(s,y;t,x)G(x,s)\,dx\,dy
\ = \ \left\langle U^{D}(t,s)G(\cdot ,s),f(\cdot )\right\rangle _{D}.
\end{eqnarray*}
Therefore, from \eqref{sept-1},
$\left\langle v(\cdot ,t),f(\cdot )\right\rangle _{D}=\left\langle
U^{D}(t,0)v(\cdot ,0),f(\cdot )\right\rangle _{D}+\int_{0}^{t}\langle
U^{D}(t,s)G(\cdot ,s),f(\cdot )\rangle _{D}\,ds
$
{\color{black} for all $f\in{\rm Dom}(L^*)$.
Since $\dom(L^{\ast })$ is dense in $L^{2}(D)$ }
 we obtain that $v$ is a
mild solution of \eqref{2.2} on $[0,T].$

To prove the converse let $v$ be a mild solution of \eqref{2.2} on $[0,T].$ For $f \in \dom(L^{\ast }),$
\begin{eqnarray}\nonumber
\lefteqn{\int_{0}^{t}\LL\langle v(\cdot ,s),\frac{1}{2}k^{2}(s)L^{\ast }f(\cdot
)\RR\rangle _{D}\,ds} \\ \nonumber
&=&\int_{0}^{t}\LL\langle U^{D}(s,0)v(\cdot ,0),\frac{1}{2}k^{2}(s)L^{\ast
}f(\cdot )\RR\rangle _{D}\,ds \\ \nonumber
&&+\int_{0}^{t}\left\langle \int_{0}^{s}\chi _{\lbrack
0,s]}(r)U^{D}(s,r)G(\cdot ,r)\,dr,\frac{1}{2}k^{2}(s)L^{\ast }f(\cdot
)\right\rangle _{D}\,ds \\ \nonumber
&=&\int_{0}^{t}\LL\langle v(\cdot ,0),(U^{D})^{\ast }(s,0)\frac{1}{2}%
k^{2}(s)L^{\ast }f(\cdot )\RR\rangle_{D}\,ds \\
&&
+\int_{0}^{t}\int_{r}^{t}\left\langle U^{D}(s,r)G(\cdot ,r),\frac{1}{2}%
k^{2}(s)L^{\ast }f(\cdot )\right\rangle _{D}\,ds\,dr. \label{6}
\end{eqnarray}
By applying the Kolmogorov forward equation to $(U^{D})^{\ast }$ we get for the
first integral on the right side of \eqref{6}:
\begin{eqnarray*}
&&
(U^{D})^{\ast }(s,0)(\frac{1}{2}k^{2}(s)L^{\ast }f)(x)=\int_{D}p^{D^{\ast }}(0,x;s,y)%
\frac{1}{2}k^{2}(s)L^{\ast }f(y)\,dy\\
&&
=\int_{D}(\frac{1}{2}k^{2}(s)L)p^{D^{\ast }}(0,x;s,y)f(y)\,dy=\int_{D}\frac{\partial }{%
\partial s}p^{D^{\ast }}(0,x;s,y)f(y)\,dy,
\end{eqnarray*}
and therefore%
\begin{eqnarray*}\lefteqn{
\int_{0}^{t}\LL\langle v(\cdot ,0),(U^{D})^{\ast }(s,0)(\frac{1}{2}k^{2}(s)L^{\ast })f(\cdot
)\RR\rangle_{D}\,ds}\\
&=&\int_{0}^{t}\LL\langle v(\cdot ,0),\int_{D}\frac{\partial }{\partial s}%
p^{D^{\ast }}(0,\cdot ;s,y)f(y)\,dy\RR\rangle_{D}\,ds \
= \ \LL\langle v(\cdot
,0),\int_{D}p^{D^{\ast }}(0,\cdot ;t,y)f(y)dy-f(\cdot )\RR\rangle _{D} \\
&=&\LL\langle v(\cdot ,0),(U^{D})^{\ast }(t,0)f(\cdot )\RR\rangle_{D}-\langle
v(\cdot ,0),f(\cdot )\rangle_{D}.
\end{eqnarray*}
In the same way we get for the second integral on the right side of \eqref{6}%
\begin{eqnarray*}
\left\langle U^{D}(s,r)G(\cdot ,r),\frac{1}{2}k^{2}(s)L^{\ast }f(\cdot
))\right\rangle _{D}
&=&\left\langle G(\cdot ,r),(U^{D})^{\ast }(s,r)(\frac{1}{2}k^{2}(s)L^{\ast
}f)(\cdot )\right\rangle _{D} \\
&=&\left\langle G(\cdot ,r),\int_{D}\frac{\partial }{\partial s}%
p^{D^{\ast }}(r,\cdot ;s,y)f(y)dy\right\rangle _{D},
\end{eqnarray*}
and therefore
\begin{eqnarray*}\lefteqn{
\int_{r}^{t}\left\langle U^{D}(s,r)G(\cdot ,r),\frac{1}{2}k^{2}(s)L^{\ast
}f(\cdot )\right\rangle _{D}ds=\int_{r}^{t}\left\langle G(\cdot ,r),\int_{D}%
\frac{\partial }{\partial s}p^{D^{\ast }}(r,\cdot ;s,y)f(y)dy\right\rangle _{D}ds}\\
&=&\left\langle G(\cdot ,r),\int_{D}p^{D^{\ast }}(r,\cdot ;t,y)f(y)dy-f(\cdot
)\right\rangle _{D}=\left\langle G(\cdot ,r),(U^{D})^{\ast }(t,r)f(\cdot
)-f(\cdot )\right\rangle _{D} \\
&=&\LL\langle U^{D}(t,r)G(\cdot ,r),f(\cdot )\RR\rangle_{D}-\LL\langle
G(\cdot ,r),f(\cdot )\right\rangle _{D}.
\end{eqnarray*}
In this way we obtain%
\begin{eqnarray*}\lefteqn{
\int_{0}^{t}\langle v(\cdot ,s),\frac{1}{2}k^{2}(s)L^{\ast }f(\cdot
)\rangle _{D}\,ds }\\
&=&\LL\langle U^{D}(t,0)v(\cdot ,0)+\int_{0}^{t}U^{D}(t,r)G(\cdot ,r)dr,f(\cdot
)\rangle _{D}-\langle v(\cdot ,0),f(\cdot )\RR\rangle_{D}
-\int_{0}^{t}\langle G(\cdot ,r),f(\cdot )\rangle _{D}\,dr \\
&=&\left\langle v(\cdot ,t),f(\cdot )\right\rangle _{D}-\langle v(\cdot
,0),f(\cdot )\rangle \,_{D}-\int_{0}^{t}\langle G(\cdot ,r),f(\cdot )\rangle
_{D}\,dr,
\end{eqnarray*}%
since $v$ is a mild solution on $[0,T]$. It follows that $v$ is a weak
solution on $[0,T].$
\end{proof}

\begin{coro}
\label{CORO1}The equations \eqref{2.1} and \eqref{2.2} possess unique weak
solutions.
\end{coro}

\begin{proof} Theorem \ref{THM8} and Proposition \ref{PROP3} show the existence and uniqueness of\ a local
weak and mild solution of \eqref{2.2}, and Proposition \ref{PROP1} shows the uniqueness of a weak solution of \eqref{2.1}.
\end{proof}

\begin{rem}
We refer to \cite{NV} for an existence and uniqueness theorem of the
variational solution of an SPDE with a nonautonomous second order differential operator and driven by fractional Brownian motion, and to \cite{RalSche} for the existence and uniqueness of the mild solution. In \cite{MRS} the existence and uniqueness of the mild solution is shown for equations with the same differential operator and driven by mixed noise.
\end{rem}

\section{An upper bound for the blowup time and probability estimates}\label{Section3}
\subsection {An upper bound for the blowup time}\label{Subsection3.1}
In  the remaining part of the paper we will assume that $L$ and $L^*$ admit strictly positive eigenfunctions: there
exists a positive eigenvalue $\lambda _{0}$ and strictly positive
eigenfunctions $\psi _{0}\in $ $\mathrm{Dom}(L)$ for $P_{t}^{D}$ and $%
\varphi _{0}\in \mathrm{Dom}(L^{\ast })$ for $(P^{D})_{t}^{\ast }$ with $%
\int_{D}\psi _{0}(x)dx=\int_{D}\varphi _{0}(x)dx=1$ such that%
\begin{equation}  \label{P}
(P_{t}^{D}-e^{-\lambda _{0}t})\psi _{0}=((P^{D})_{t}^{\ast }-e^{-\lambda
_{0}t})\varphi _{0}=0,
\end{equation}
{\color{black}hence}
\begin{equation}  \label{L}
(L+\lambda _{0})\psi _{0}=(L^{\ast }+\lambda _{0})\phi _{0}=0.
\end{equation}%
 For generators of a
general class of L\'{e}vy processes, properties (\ref{P}) and (\ref{L})
follow from \cite{KyS,CyW}. 
Another example are the diffusion processes: for $f\in
\mathcal{C}_{0}^{2}(D),$ the set of twice continously differentiable
functions with compact support in $D,$ let us define the differential
operator
\begin{equation*}
Lf=\sum_{j,k=1}^{d}\frac{\partial }{\partial x_{j}}\left(a_{jk}\frac{%
\partial }{\partial x_{k}}f\right)+\sum_{j=1}^{d}b_{j} \frac{\partial }{%
\partial x_{j}}f-cf,
\end{equation*}
where $a_{j,k},$ $b_{j},$ $j,k=1,...,d$ are bounded smooth functions on $D$
and $c$ is bounded and continous. We assume that the matrix $(a_{j,k},$ $%
j,k=1,...,d)$ is symmetric and uniformly elliptic.
 In this case properties (\ref{P}) and (\ref{L}) follow from \cite[%
Theorem 11, Chapter 2]{Friedman}.


\begin{theorem}
\label{THM1}
Assume (\ref{L}) and let $g(z) \geq Cz^{1+\beta }$ for all $z>0$, where $C>0$, $\beta >0$, are given constants.
Let us define
\begin{equation}
\tau ^{\ast }=\inf \left\{t>0:\int_{0}^{t}\exp \left[-\beta (\lambda
_{0}K(r)+A(r))+\beta N_{r}\right]\,dr\ \geq \ \frac{1}{C\beta }\langle {%
	\varphi ,\phi _{0}\rangle}_{D}^{-\beta }\right\},  \label{tau*}
\end{equation}
where the functions $K$ and $A$ are defined in (\ref{AK}). Then, on the
event $\{\tau ^{\ast }<\infty \}$ the solution $v$ of~\eqref{2.2} and the
solution $u$ of \eqref{2.1} blow up in finite time $\tau$, and $\tau \leq
\tau ^{\ast}$ $\mathbb{P}$-a.s.
\end{theorem}

\begin{proof} 
Using the hypothesis on $g$ and   Jensen's inequality we get for the terms in \eqref{weak rpde}:
\begin{eqnarray*}
	\langle v(\cdot ,s),L^{\ast }\phi _{0}\rangle _{D}&=&-\lambda _{0}\langle
	v(\cdot ,s),\phi _{0}\rangle _{D},\\
	\exp (-N_{s})\left\langle g(\exp (N_{s})v(\cdot ,s)),\phi _{0}\right\rangle
	_{D} &\geqq &C\exp (\beta N_{s})\left\langle v^{1+\beta }(\cdot ,s),\phi
	_{0}\right\rangle _{D} ,\\
	&\geqq &C\exp (\beta N_{s})\langle {v(\cdot ,s),\phi _{0}\rangle }%
	_{D}^{1+\beta }.
\end{eqnarray*}%
Applying these lower bounds to ($\langle {v(\cdot ,t+\varepsilon ),\phi
	_{0}\rangle }_{D}-\langle {v(\cdot ,t),\phi _{0}\rangle }_{D})/\varepsilon $
and letting $\varepsilon \rightarrow 0$ we get
\begin{equation}
\frac{d}{dt}\langle {v(\cdot ,t),\phi _{0}\rangle }_{D}\geqq -\frac{1}{2}%
(\lambda _{0}k^{2}(t)+a^{2}(t))\langle v(\cdot ,t),\phi _{0}\rangle
_{D}+C\exp (\beta N_{t})\langle {v(\cdot ,t),\phi _{0}\rangle }_{D}^{1+\beta
}.  \label{ineq}
\end{equation}
The corresponding differential equality reads%
\begin{equation*}
\frac{d}{dt}I(t)=-\frac{1}{2}(\lambda _{0}k^{2}(t)+a^{2}(t))I(t)+C\exp
(\beta N_{t})I(t)^{1+\beta },
\end{equation*}
and  $I(t)$ is  a subsolution of
\eqref{ineq}, i.e. $\langle {v(\cdot ,t),\phi _{0}\rangle }_{D}\geqq I(t)$.
Then
\begin{equation*}
I(t)=\exp [-(\lambda _{0}K(t)+A(t))]\left( \langle \varphi {,\phi
	_{0}\rangle }_{D}^{-\beta }-\beta C\int_{0}^{t}\exp\LL [-\beta (\lambda
_{0}K(s)+A(s))+\beta N_{s}\RR]ds\right) ^{-1/\beta }
\end{equation*}
for all $t\in \lbrack 0,\tau ^{\ast }),$ where $\tau ^{\ast }$ is given by~%
\eqref{tau*}. Therefore $\tau ^{\ast }$ is an upper bound for the blowup
time of $\langle {v(\cdot ,t),\phi _{0}\rangle }_{D}$, and the function
$
t\mapsto \|{v(\cdot ,t)\| }_{\infty }=\exp
(-N_{t})\| {u(\cdot ,t)\|}_{\infty }
$
can not stay finite on $[0,\tau ^{\ast }]$ if $\tau ^{\ast }<\infty $.
Therefore $u$ and $v$ blow up before $\tau ^{\ast }$ if $\tau ^{\ast
}<\infty $.
\end{proof}

\begin{rem}
Notice that $\tau ^{\ast }$ depends on $L$ only by the positive
	eigenvalue $\lambda _{0}$ and the associated eigenfunction $\phi _{0}.$
	Moreover, $\tau ^{\ast }$ is a decreasing function of $\varphi ,$ $\phi _{0}$
	and $C$, and an increasing function of $\lambda _{0}K.$ Therefore small
	functions $\varphi $, $\phi _{0}$ and a small constant $C,$ as well as high
	values of $\lambda _{0}K$ postpone the blowup of $I$ and have, in this
	sense, the tendency to postpone the blowup of $v$ and $u.$
\end{rem}

\subsection {A tail probability estimate for the upper bound of the blowup time} \label{section.3.1}

In the following theorem  we apply a tail probability estimate for exponential
functionals of fBm studied by N.T.\ Dung \cite{D} to estimate the
probability that $\tau ^{\ast }$ occurs before a fixed time $T$. Here we assume that the process $B^H$ is given by the formula
\begin{equation}\label{FBM}
B_{t}^{H}=\int_{0}^{t}{K^{H}(t,s)\,dB_{s}},
\end{equation}
where  the kernel $K^{H}$ is given for $H>1/2$ by
\begin{equation}\label{Kernel-K}
K^{H}(t,s)=\left\{
\begin{array}{ll}
C_{H}s^{1/2-H}\int_{s}^{t}{(\sigma -s)^{H-3/2}\sigma ^{H-1/2}d\sigma } &
\text{ if }t>s, \\
&  \\
0 & \text{ if }t\leqq s,%
\end{array}%
\right.
\end{equation}%
{\color{black}
where $C_{H}=[\frac{H(2H-1)}{\mathcal{B}(2-2H,H-1/2)}]^{\frac{1}{2}}$ and $\mathcal{B}$ is the usual beta function (see Section 5.1.3 in   \cite {N} for a general representation formula of fBm with $H>1/2$). Notice that $B^H$ and $B$ are dependent in this case.}

\begin{theorem} \label{THM2} Under assumptions \eqref{L} and \eqref{FBM},
let
$g(z) \geq Cz^{1+\beta
}$
for
all $z>0$, where $C>0$, $\beta >0$, are given constants,
and let
$\mu(T) =\int_{0}^{T}\exp [-\beta (\lambda _{0}K(t)+A(t))]%
\mathbb{E}\left[ \exp (\beta N_{t})\right] dt $. Then, for any $T>0$ such that $\frac{1}{C\beta }\langle \varphi ,\phi
_{0}\rangle _{D} ^{-\beta } >\mu(T) ,$
\begin{equation*}
\mathbb{P}\left\{ \tau ^{\ast }\leq T\right\} \leq 2\exp \left( -\frac{\ln^2%
	\left[ C\beta\langle \varphi ,\phi _{0}\rangle_{D}^{\beta} \, \mu(T) \right]%
}{2M(T)}\right) ,
\end{equation*}%
where
\begin{equation*}
M(T) =2\beta ^{2}\int_{0}^{T}a^{2}(r)\,dr +4\beta ^{2}HT^{2H-1}\int_{0}^{T}b^{2}(u)\,du.
\end{equation*}

\end{theorem}
\begin{proof} For $t\ge0, $ using \eqref{FBM}, we have the following representation:
\begin{eqnarray}
X_{t}&:=&-\beta(\lambda_0 K(t)+A(t)) +\beta N_{t} \\ \nonumber \label{Rep}
&=&-\beta(\lambda_0 K(t)+A(t)) +\beta \left(
\int_{0}^{t}a(s)\,dB_{s}+\int_{0}^{t}\int_{s}^{t}b(r)\frac{\partial }{\partial r}
K^H(r,s)\,dr\,dB_{s}\right).\\ \nonumber
\end{eqnarray}%
From  \cite[Theorem 3.1]{D}
it follows that for any $T\geq 0$ and any $x>\mu(T),$ there holds
\begin{equation}
\PP\left( \int_{0}^{T}e^{X_{t}}dt\geq x\right) \leq 2\exp \LL[-\frac{(\ln x-\ln
	\mu(T) )^{2}}{2M(T)}\RR],  \label{Dung}
\end{equation}%
where $\mu (T)=\int_{0}^{T}\mathbb{E}\left[
e^{X_{t}}\right] dt$ and $M(T)$ is such that
\begin{equation}\label{M1}
\sup_{t\in \lbrack 0,T]}\int_{0}^{T}|D_{r}X_{t}|^{2}\,dr\leq M(T)\quad
\mbox{$\PP$-a.s.}
\end{equation}
Here $D_rX_t$ denotes the Malliavin derivative of $X_t$.
In the following we will find an upper bound $M(T)$ such that (\ref{M1}) holds. For $r<t$ we have, using the representation \eqref{Rep},
\begin{eqnarray*}
	D_{r}X_{t}=\beta \left( a(r)+\int_{r}^{t}b(s)\frac{\partial }{\partial s}
	K(s,r)\,ds \right) .\\
\end{eqnarray*}
Hence  $\int_0^t |D_{r}X_{t}|^2 \,dr \leq 2 \beta^2 \int_0^t a^2(r) \,dr + 2 \beta^2 \int_0^t(\int_r^t b(s)  \frac{\partial}{\partial s}K(s,r) \,ds)^2 \,dr $ and
\begin{eqnarray*}
	\lefteqn{
		\int_0^t\left(\int_r^t b(s) \frac{\partial}{\partial s}K(s,r) \,ds\right)^2 \,dr}\\
	&=&\int_0^t \left(\int_r^t b(s)  \frac{\partial}{\partial s}K(s,r) \,ds\right)\left(\int_r^t b(s')  \frac{\partial}{\partial s'}K(s',r) \,ds'\right)\,dr  \\ \
	&=&\int_0^t  b(s)\,ds \int_0^s  \frac{\partial}{\partial s}K(s,r) \,dr \int_r^t b(s')  \frac{\partial}{\partial s'}K(s',r) \,ds' \\ \
	&=&\int_0^t ds\, b(s) \int_0^t \,dr 1_{[0,s]} (r) \frac{\partial}{\partial s}K(s,r) \int_r^t b(s')  \frac{\partial}{\partial s'}K(s',r) \,ds' \\ \
	&=&\int_0^t ds\, b(s) \int_0^t \,ds' b(s') \int_0^{s'}1_{[0,s]} (r) \frac{\partial}{\partial s}K(s,r)  \frac{\partial}{\partial s'}K(s',r) \,dr \\ \
	&=&\int_0^t \,ds  \int_0^t ds' \, b(s) b(s') \int_0^{s \wedge s'}  \frac{\partial}{\partial s}K(s,r) \frac{\partial}{\partial s'}K(s',r) \,dr\\ \
	&=&\int_0^t \,ds  \int_0^t ds' \, b(s) b(s') \Phi(s,s')\\ \
	&=&\int_0^t \,ds  \int_0^s ds' \, b(s) b(s') \Phi(s,s')+\int_0^t \,ds  \int_s^t ds' \, b(s) b(s') \Phi(s,s') \\ \
	&=&2 \int_0^t \,ds  \int_0^s ds' \, b(s) b(s') \Phi(s,s'),
\end{eqnarray*}
where $\Phi(s,s')=\int_0^{s \wedge s'}  \frac{\partial}{\partial s}K(s,r) \frac{\partial}{\partial s'}K(s',r) \,dr.$
Since  $ \frac{\partial}{\partial s} K(s,r)=C_{H}r^{1/2-H} (s-r)^{H-3/2}s^{H-1/2},$ using (5.7) in \cite{N} we obtain
\begin{eqnarray*}
\Phi(s,s')= C_{H}^{2}(ss')^{H-1/2} \int_0^{s \wedge s'} r^{1-2H}(s-r)^{H-3/2}(s'-r)^{H-3/2} \,dr
= H(2H-1)(s-s')^{2H-2}
\end{eqnarray*}
for $s'<s,$  hence
\begin{align}\nonumber
\lefteqn{\int_0^t\left(\int_r^t b(s)  \frac{\partial}{\partial s}K(s,r) \,ds\right)^2 \,dr} \\ \nonumber
&\leq \ 2H(2H-1)\int_0^t \,ds\int_0^s
|b(s)b(s')|(s-s')^{2H-2} \,ds' \\ \nonumber
&\leq\ H(2H-1)\left[
\int_0^t  b(s)^2 \int_0^s (s-s')^{2H-2} \,ds'
\,ds
+ \int_0^t \int_0^s b(s')^2 (s-s')^{2H-2}\,ds'
\,ds
\right] \\ \nonumber
&=\ H \int_0^t  b(s)^2  s^{2H-1}\,ds + H(2H-1) \int_0^t  b(s')^2 \int_{s'}^t (s-s')^{2H-2} \,ds\, ds' \\ \nonumber
&=\ H \int_0^t  b(s)^2  (s^{2H-1}+(t-s)^{2H-1}) \,ds\\ \label{R}
& \leq \ 2H t^{2H-1} \int_0^t  b(s)^2   \,ds.
\end{align}
From the above inequalities we obtain
\begin{equation}
\sup_{t\in \lbrack 0,T]}\int_{0}^{T}|D_{r}X_{t}|^{2}dr
\leq 2\beta ^{2}\int_{0}^{T}a^{2}(r)dr
+4\beta ^{2}HT^{2H-1}\int_{0}^{T}b^{2}(u)du:=M(T).
\label{M2}
\end{equation}
Now, from (\ref{tau*})
\begin{eqnarray}\nonumber
\mathbb{P}(\tau ^{\ast }\leqq T)&=&\mathbb{P}\left\{ \int_{0}^{T}\exp
[-\beta (\lambda _{0}K(t)+A(t))+\beta N_{t}]\,dt \geqq  \frac{1}{C\beta}\langle \varphi ,\phi _{0}\rangle _{D}^{-\beta} \right\} \\ \label{X}
&=&\mathbb{P}\left\{ \int_{0}^{T}\exp[X(t)] \,dt \geq x  \right\},\end{eqnarray} where $x= \frac{1}{C\beta}\langle \varphi ,\phi _{0}\rangle _{D}^{-\beta} .$ The result follows from  \eqref{Dung} and \eqref{M2}.\end{proof}

{\color{black}
In the following theorem we obtain upper bounds for the tail of $\tau^*$ in the case {\color{black} when the Brownian motion $B$} and the fractional Brownian  motion $B^H$  have general dependence structure.

\begin{theorem}\label{THM3}
Assume \eqref{L} and let $g(z) \geq Cz^{1+\beta }$ for all $z>0$, where $C>0$, $\beta >0$, are given constants.
\begin{enumerate}
\item Assume that $B_{t}^{H}=\int_{0}^{t}{K^{H}(t,s)\,dW_{s}},$ where $W$ is a Brownian motion defined in the same probability space, and adapted to the same filtration as the Brownian motion $B$.
	  Then
	\begin{eqnarray*}\lefteqn{
	\PP(\tau^*\le T)} \\
	 &\le & C\beta \langle {%
				\varphi ,\phi _{0}\rangle}_{D}^{\beta }
\int_0^T\LL[
			e^{-\beta \lambda_0 \int_0^t k^2(s)\,ds + 2\beta^2\int_0^ta^2(s)\,ds}
			+
			e^{ - \beta \int_0^t a^2(s)\,ds
				+4\beta^2Ht^{2H-1}\int_0^tb^2(s)\,ds}
			\RR]\,dt.
	\end{eqnarray*}

	\item If $B$ and $B^H$ are independent, then {\color{black}
	$$\PP(\tau^*\le T) \ \le \ { C\beta \langle {%
			\varphi ,\phi _{0}\rangle}_{D}^{\beta }}
					\int_0^T
			e^{-\beta\lambda_0K(t)
				+\frac{\beta^2-\beta}{2}\int_0^ta^2(s)\,ds
				+ \beta^2H t^{2H-1}\int_0^tb^2(s)\,ds}.
			$$
	}
\end{enumerate}
\end{theorem}

\begin{proof}
\begin{enumerate}
	\item Using Hölder's and Chebishev's inequalities we obtain
	{\color{black}
	\begin{eqnarray}\nonumber
		\PP(\tau^*\le T) &=& \PP\LL[
		\int_0^T
		e^{-\beta\lambda_0K(t)
			+\beta\int_0^ta(s)\,dB_s -\beta A(t)+\beta \int_0^tb(s)\,dB^H_s}
		\,dt\ge  \frac{1}{C\beta }\langle {%
			\varphi ,\phi _{0}\rangle}_{D}^{-\beta }
		\RR]
		\\ \nonumber
		&\le &
		\PP\LL[\LL(
		\int_0^T
		e^{-2\beta\lambda_0K(t)
			+2\beta\int_0^ta(s)\,dB_s }
		\,dt\RR)^{\frac{1}{2}} \right. \\ \nonumber &&\phantom{MMMM}\times\left.
		\LL(
		\int_0^T
		e^{ -2\beta A(t)+2 \beta \int_0^tb(s)\,dB^H_s}
		\,dt\RR)^{\frac{1}{2}}
		\ge  \frac{1}{C\beta }\langle {%
			\varphi ,\phi _{0}\rangle}_{D}^{-\beta }
		\RR]
		\\ \nonumber
		&\le &
		\PP\LL[
		\int_0^T
		e^{-2\beta\lambda_0K(t)
			+2\beta\int_0^ta(s)\,dB_s }
		\,dt\ge  \frac{1}{C\beta }\langle {%
			\varphi ,\phi _{0}\rangle}_{D}^{-\beta }\RR] \\ \nonumber &&
		+
		\PP\LL[
		\int_0^T
		e^{-2\beta A(t)+ 2\beta \int_0^tb(s)\,dB^H_s}
		\,dt\ge  \frac{1}{C\beta }\langle {%
			\varphi ,\phi _{0}\rangle}_{D}^{-\beta }
		\RR]
			\\ \nonumber
		&\le & \frac{
			\EE\LL[
			\int_0^T
			e^{-2\beta\lambda_0K(t)
				+2\beta\int_0^ta(s)\,dB_s }
			\,dt\RR]
			+
			\EE\LL[
			\int_0^T
			e^{-2\beta A(t)+ 2\beta\int_0^tb(s)\,dB^H_s}
			\,dt
			\RR]}{ \frac{1}{C\beta }\langle {%
				\varphi ,\phi _{0}\rangle}_{D}^{-\beta }}\\ \nonumber
					 				&\le &\frac{\int_0^T
			\left[e^{-2\beta\lambda_0K(t)
				+2\beta^2\int_0^ta^2(s)\,ds}\right]\,dt
			+
			\int_0^T
			e^{-2\beta A(t)}
			\EE\LL[e^{ 2\beta\int_0^t b(s)\,dB^H_s}\RR] \,dt
				}{ \frac{1}{C\beta }\langle {%
				\varphi ,\phi _{0}\rangle}_{D}^{-\beta }},\\ \label{B1B2B3}
		\end{eqnarray}			
}
	where  we have used   the fact that $\EE\LL(\exp\LL\{\int_0^tf(s)\,dB(s)\RR\}\RR)=
\exp\LL\{\frac{1}{2}\int_0^tf^2(s)\,ds\RR\}$ to obtain the last inequality.
{\color{black} In addition,
$$
\EE\LL[e^{ 2\beta\int_0^t b(s)\,dB^H_s}\RR]
=
\EE\LL[e^{ 2\beta\int_0^t\int_s^t b(r)\frac{\partial}{\partial r}K^H(r,s)\,dr\,dW_s}\RR]  =
e^{ 2\beta^2\int_0^t \LL[\int_s^t b(r)\frac{\partial}{\partial r}K^H(r,s)\,dr\RR]^2\,ds},
$$
where  the last equality follows from \cite[Theorem 4.12]{Klebaner}. Therefore, using \eqref{R} we get
\begin{equation}\label{B3B3B3}
\EE\LL[e^{ 2\beta\int_0^t b(s)\,dB^H_s}\RR]\le \exp\LL\{
4\beta^2Ht^{2H-1}\int_0^tb^2(s)\,ds
\RR\}.
\end{equation}
Substituting \eqref{B3B3B3} into \eqref{B1B2B3} we obtain the desired bound.
}
	\item  Using Chebishev's inequality,  the independence of $B$ and $B^H$
	{\color{black}
	and the proof of \eqref{B3B3B3},
	\begin{eqnarray*}\lefteqn{
		\PP(\tau^*\le T)} \\ &=& \PP\LL[
		\int_0^T
		e^{-\beta\lambda_0K(t)
			+\beta\int_0^ta(s)\,dB_s -\beta A(t)+\beta \int_0^t b(s)\, dB^H_s)}
		\,dt\ge  \frac{1}{C\beta }\langle {%
			\varphi ,\phi _{0}\rangle}_{D}^{-\beta }
		\RR]
		\\
		&\le & C\beta \langle {%
				\varphi ,\phi _{0}\rangle}_{D}^{\beta }
			\int_0^T
		\EE\LL[e^{-\beta\lambda_0K(t)
				+\beta\int_0^t a(s)\,dB_s}\RR]
			\EE\LL[e^{	-\beta A(t)+ \beta \int_0^t b(s)\, dB^H_s }\RR]
			\,dt \\
		& \le & C\beta \langle {%
				\varphi ,\phi _{0}\rangle}_{D}^{\beta }
			\int_0^T
			\exp\LL\{
				-\beta\lambda_0K(t)
				+\frac{\beta^2-\beta}{2}\int_0^ta^2(s)\,ds
				+
				 \beta^2Ht^{2H-1}\int_0^tb^2(s)\,ds
\RR\}
			\,dt.
	\end{eqnarray*}
	}
\end{enumerate}
\end{proof}
}

\section{Lower bounds for the blowup time and for the probability of finite
time blowup}
\subsection{A lower bound for the probability of finite
time blowup}
In the following theorem we give a lower bound for the probability of finite
time blow up of the {\color{black} weak} solution of (\ref{2.1}). If $f,g$ are nonnegative
functions and $c$ is a constant, we write $f(t)\sim cg(t)$ as $t\to\infty$
if $\lim_{t\to\infty}f(t)/g(t)=c$.
{\color{black}
\begin{theorem}\label{THM4} {\color{black} Assume \eqref{L} and \eqref{FBM}.
		 Let
 $g(z)\geq Cz^{1+\beta }$
}
and
\begin{equation*}
\int_{0}^{t}a^{2}(r)\,dr\sim C_{1}t^{2l},\quad \int_{0}^{t}b^{2}(r)\,dr\sim C_{2}t^{2m},\quad
 \int_{0}^{t}k^{2}(r)\,dr\sim C_{3}t^{2p}
\end{equation*}%
as $t\rightarrow \infty $ for some  nonnegative constants $l,\,m,\, p$ and positive
constants $C,$ $\beta$, $C_{1}, C_2$ and $C_{3}.$ Suppose additionally that
\begin{enumerate}{\color{black}
\item if $\beta \in (0,1/2),$ then $\max \{p,l\}>H+m-\frac{1}{2},
$
\item if $\beta =1/2,$ then H+$m-\frac{1}{2}<p,$

\item if $\beta >1/2,$
then $p>\max \{l,H+m-\frac{1}{2}\}.$
}
\end{enumerate}
Under these assumptions the solution
of (\ref{2.1}) blows up in finite time with positive probability. Moreover,
\begin{equation}
\mathbb{P}(\tau <\infty )\ \geqq \ \mathbb{P}(\tau ^{\ast }<\infty )\ \geqq
\ 1-\exp \left( -\frac{(m_{\xi }-1)^{2}}{2L_{\xi }}\right) ,  \label{DV}
\end{equation}%
where 
\begin{equation}
\xi =\frac{1}{C\beta }\langle \varphi ,\phi _{0}\rangle _{D}^{-\beta },\quad
L_{\xi }=\underset{t\geqq 0}{\sup }\frac{M(t)}{(\ln (\xi +1)+f(t))^{2}},
\end{equation}%
with $f(t)=t^{\max \{H+m-1/2,\,l\}}$ and
\begin{equation}
m_{\xi }=E\left[ \underset{t\geqq 0}{\sup }\frac{\ln \left( \int_{0}^{t}\exp
	\left( -\beta (\lambda _{0}K(s)+A(s))+\beta N_{s}\right) \,ds+1\right) +f(t)%
}{\ln (\xi +1)+f(t)}\right] .  \label{DIII}
\end{equation}
\end{theorem}

\begin{proof} From (\ref{X})  it follows that $\PP(\tau^* <\infty) =\PP( \int_0^{\infty} e^{X_t} \,dt \geq \xi).$
In order to estimate $\PP( \int_0^{\infty} e^{X_t} \,dt \geq \xi)$ we use \cite[Theorem 3.1]{DII}, with $a=0$ and $ \sigma=1:$
\begin{prop}[\cite{DII}] \label{DII}Assume that the stochastic process $X$ is adapted and satisfies
	
	a) $\int_0^{\infty}Ee^{X_s}\,ds <\infty,$
	
	b) For each $t\geq 0, \, X_t \in D^{1,2},$
	
	c) There exists a function $f:R_+\to R_+$ such that $\lim_{t \to \infty}f(t) = \infty$ and for each $x>0,$
	\begin{eqnarray}
	\text{ \ \ }\underset{t\geqq 0}{\sup }\frac{\sup_{s\in \lbrack
			0,t]}\int_{0}^{t}|D_{r}X_{s}|^{2}dr}{(\ln (x+1)+f(t))^{2}}
	\leq
	L_{x}<\infty  \quad {a.s.}
	\end{eqnarray}
	Then
	$$\PP\left( \int_0^{\infty} e^{X_t} \,dt <x\right) \leq \exp\LL\{-\frac{(m_x-1)^2}{2L_x}\RR\},$$
	where $$m_x= E\left[\sup_{t\geq 0} \frac{\ln(\int_0^t e^{X_s}\,ds +1)+f(t)}{\ln(x+1)+f(t)}\right].$$
\end{prop}

We now verify that conditions a) - c) of the above proposition hold.

For condition a) we have from \eqref{Rep},
{\color{black}
\begin{eqnarray*}\nonumber
\lefteqn{
	\int_{0}^{\infty }\EE \exp [X_{t}]\,dt}\\
&=&	\int_{0}^{\infty }\EE \exp \left[-\frac{\beta \lambda _{0}}{2}\int_0^tk^2(s)\,ds-\frac{\beta}{2}\int_0^ta^2(s)\,ds \right. \\ && \phantom{MMMMM} + \left. \beta\left( \int_0^t  a(s)\,dB_s  +\int_0^t\int_s^t b(r)\frac{\partial}{\partial r}K^H(r,s)\,dr\,dB_s\right)\right] dt\\ \nonumber
&=&	
	\int_{0}^{\infty }\EE \exp \left[-\frac{\beta \lambda _{0}}{2}\int_0^tk^2(s)\,ds-\frac{\beta}{2}\int_0^ta^2(s)\,ds + \beta \int_0^t\left(  a(s) +\int_s^t b(r)\frac{\partial}{\partial r}K^H(r,s)\,dr\right)\,dB_s\right] dt\\
	&=&	
	\int_{0}^{\infty } \exp \left[-\frac{\beta \lambda _{0}}{2}\int_0^tk^2(s)\,ds-\frac{\beta}{2}\int_0^ta^2(s)\,ds + \frac{\beta^2}{2} \int_0^t
		\left(  a(s) +\int_s^t b(r)\frac{\partial}{\partial r}K^H(r,s)\,dr\right)^2\,ds\right] dt,
	\end{eqnarray*}
where, again,  we have used \cite[Theorem 4.12]{Klebaner} to obtain the last equality. Therefore, using \eqref{R},
\begin{eqnarray}
	\nonumber
		\int_{0}^{\infty }\EE \exp [X_{t}]\,dt
	&\le&
\int_{0}^{\infty } \exp \left[-\frac{\beta \lambda _{0}}{2}\int_0^tk^2(s)\,ds-\frac{\beta}{2}\int_0^ta^2(s)\,ds + \frac{\beta^2}{2} \int_0^t
  2a^2(s)\,ds\right. \\ \label{KlKl} & & \phantom{MMMMM}  + \left.
2\beta^2Ht^{2H-1}\int_0^tb^2(s)\,ds
\right] dt.
	\end{eqnarray}
}
The  integral \eqref{KlKl} is finite if and only if the leading power of $t$ in the term
$$
-\frac{\beta \lambda_{0}}{2} \int _0^t k^2(s)\, ds +\frac{2\beta^2 -\beta}{2} \int_0^t a^2(s)\,ds + 2\beta^2Ht^{2H-1}\int_0^tb^2(s)\,ds
$$
has negative coefficient, which follows from our assumptions.

Condition b) is a consequence of  \eqref{M2}.

For condition c) we use the inequality  (\ref {M2}), which implies that for any $x>0$ and any fixed function $f$,
\begin{eqnarray}\label{sup}
\text{ \ \ }\underset{t\geqq 0}{\sup }\frac{\sup_{s\in \lbrack
		0,t]}\int_{0}^{t}|D_{r}X_{s}|^{2}dr}{(\ln (x+1)+f(t))^{2}}
\leq \sup_{t\geq 0} \frac{M(t)}{(\ln (x+1)+f(t))^{2}}.
\end{eqnarray}%
Due to our assumptions, for big $t$,  the leading power of $t$ in the numerator is $ \max \{2l,2H +2m -1\}.$ It follows that
$$\lim_{t\to \infty}\frac{M(t)}{\LL(\ln (x+1)+t^{\max\{l,H+m-1/2 \}}\RR)^{2}} <\infty, $$
and therefore the supremum in (\ref{sup}) is finite. The result follows from Proposition \ref{DII}.
\end{proof}

The cases when $a=0$ (presence only of fractional Brownian motion) or $b=0$
(presence only of Brownian motion), are simpler:

\begin{coro} {\color{black} Under the assumptions in Theorem \ref{THM4},}

\begin{enumerate}
	\item When $a(t)\equiv 0$ and {\color{black} $p>H+m-1/2$} the solution of (\ref{2.1})
	explodes in finite time with positive probability for all $\beta >0$.
	
	\item If $a(t)\equiv 0$ and {\color{black}$p=H+m-1/2$}, the solution of (\ref{2.1})
	explodes in finite time with positive probability for all $\beta >0$
	satisfying {\color{black}$
	\,\beta < \frac{C_3\lambda_0}{4 C_2H} 
	.$
}
	
	\item When $b(t)\equiv0$ and {\color{black}$0<\beta \leq\frac{1}{2}$} the solution of (\ref{2.1}) exhibits
	explosion in finite time with positive probability for all values of $p$ and
	$l$.
	
	\item If $b(t)\equiv0$ and {\color{black}$\beta >1/2,$} the solution of (\ref{2.1}) exhibits
	explosion in finite time with positive probability if $p>l$ or if $p=l$ and {\color{black}$%
	C_{3}\lambda _{0}>C_{1}(2\beta -1)$.}
\end{enumerate}
\end{coro}
}

Notice that $m_{\xi}$ given in (\ref{DIII}) satisfies  $m_{\xi }>1$ due to
Theorem 3.1 in \cite{DII}. The formula for $m_{\xi }$ shows interactions
between $\varphi $ and $K$ that have an influence on the lower bound in %
\eqref{DV}. 
Increasing values of $K$ decrease the lower bound in \eqref{DV}. In this
sense high values of $K$ are in favour of absence of finite time blowup.

{\color{black}
\subsection{The case $H>3/4$ and independent $B$ and $B^H$}
}
In order to find more explicit lower bounds for $\mathbb{P}(\tau <+\infty )$%
, we consider in this subsection the case $H\in (3/4,1)$ and
suppose that $B$ and $B^{H}$ are independent and $b(s)=ca(s)$ for all $%
s\geqq 0$, where $c$ is a constant. Then $N_{t}=\int_{0}^{t}a(s)dM_{s}$ with
$M_{s}=B_{s}+cB_{s}^{H}$. By \cite{Ch} $M$ is equivalent to a Brownian
motion $\widetilde{B}$, and therefore $N_{t}$ is equivalent to $\tilde{N}%
_{t}:=\int_{0}^{t}a(s)\,d\widetilde{B}_{s}$. Here equivalence means equality
of the laws of the processes on ($\mathcal{C}[0,T],\mathcal{B}),$ the space
of continous functions defined on $[0,T]$ endowed with the $\sigma -$algebra
generated by the cylinder sets. Furthermore, ($\tilde{N}_{t})_{t\geqq 0}$ is
a continous martingale and therefore a time-changed Brownian motion: $\tilde{%
N}_{t}=\widetilde{B}_{2A(t)}$. \newline

\begin{theorem}
\label{THM5}
{\color{black}Assume \eqref{L}}.
 Let $H \in (3/4,1)$, $B$ and $B^{H}$ be independent and $%
b(s)=ca(s)$ for all $s\geqq 0$, where $c$ is a constant.We assume also that $%
g(z)\geq Cz^{1+\beta}$, that the functions $k$ and $a$ are positive continuous on
$\mathbb{R}_{+}$ and that there exist constants $\eta \in (0,+\infty ]$ and $%
c_{1}>0$ such that
\begin{equation}
\frac{1}{a^{2}(t)}\exp(-\beta \lambda _{0}K(t)) \geq c_{1}\exp \left(-2\beta
\frac{A(t)}{\eta }\right),\quad t\in \mathbb{R}_{+}.  \label{PB}
\end{equation}%
Then
\begin{equation}
\mathbb{P}(\tau <+\infty ) \ \geq \  \mathbb{P}(Z_{\mu }\leq \theta ),
\end{equation}%
where $\tau $ is the blowup time of \eqref{2.1}, $Z_{\mu }$ is a
gamma-distributed random variable with parameter $\mu :=\frac{2}{\beta }(%
\frac{1}{\eta }+\frac{1}{2}),$ $\theta :=\frac{2c_{1}}{\beta ^{2}\xi }$ and $%
\xi :=\frac{1}{C\beta }\langle \varphi ,\phi _{0}\rangle_{D} ^{-\beta }$.
\end{theorem}

\begin{proof}
{\color{black} From Theorem \ref{THM1},}
\begin{eqnarray*}
	\PP(\tau^{*}=+\infty)&=&\PP\left( \int_{0}^{t}dr\,\exp \left[-\beta(\lambda _{0}K(r)+A(r))+\beta \tilde{N}_{r}\right]<\xi\mbox{ for all }   t>0 \right)\\
	&=&\PP\left( \int_{0}^{\infty}dr\exp \left[-\beta(\lambda _{0}K(r) + A(r))+\beta \tilde{N}_{r}\right] \leq \xi \right).
\end{eqnarray*}
By the change of variable $q=2A(r)$ we get
\begin{equation*}
\begin{aligned}
\PP(\tau^{*}=+\infty)=\PP\left(\int_{0}^{\infty}dr\,\exp \left[-\beta(\lambda _{0}K(r)+ A(r))+\beta \tilde{B}_{2A(r)}\right] \leq \xi \right)\\
=\PP\left( \int_{0}^{\infty}\frac{dq}{a^{2}(A^{-1}(q/2))}\exp \left[-\beta(\lambda _{0}K(A^{-1}(q/2))+\frac{1}{2}q)+\beta \tilde{B}_{q}\right] \leq \xi\right).
\end{aligned}
\end{equation*}
Applying~\eqref{PB} to $t=A^{-1}(q/2)$ yields
\begin{equation*}
\frac{1}{a^{2}(A^{-1}(q/2))} \exp\left[-\beta(\lambda _{0}K(A^{-1}(q/2))\right] \geq c_{1} \exp\left(-\frac{\beta}{\eta}q\right),\quad q\in\mathbb{R}_{+}.
\end{equation*}
Therefore
\begin{eqnarray*}
	\PP(\tau^{*}=+\infty) &\leq &\PP\left(c_{1}\int_{0}^{\infty} dq \exp\left [-\beta q\left(\frac{1}{\eta}+\frac{1}{2}\right)+\beta \tilde{B}_{q}\right] \leq \xi\right)\\
	&=&\PP\left(\int_{0}^{\infty}dq\, \exp\left[\beta(\tilde{B}_{q}-\tilde{\mu} q)\right] \leq\frac{\xi}{c_{1}}\right),
\end{eqnarray*}
where $\tilde{\mu}:=\frac{1}{\eta}+\frac{1}{2}$. A second change of variable $q=\frac{4s}{\beta^{2}}$ yields
\begin{equation*}
\begin{aligned}
\PP(\tau^{*}=+\infty)
\le \PP\left(\int_{0}^{\infty} ds \,\exp\left[2(\tilde{B}_{s}-\mu s)\right] \leq \frac{\beta^{2}\xi}{4 c_{1}} \right),
\end{aligned}
\end{equation*}
where $\mu:=\tilde{\mu}\frac{2}{\beta}$.
Due to \cite[Corollary 1.2, page 95]{Y},
\begin{equation*}
\int_{0}^{\infty}e^{2(\tilde{B}_{s}-\mu s)}\,ds\overset{\mathcal{L}}{=}\frac{1}{2Z_{\mu}},
\end{equation*}
where $Z_{\mu}$ is a gamma-distributed random variable with parameter $\mu$. Therefore
\begin{equation*}
\begin{aligned}
\PP(\tau=+\infty) \leq\ \PP(\tau^{*}=+\infty) \leq \PP\left(\frac{1}{2Z_{\mu}} \leq \frac{\beta^{2}\xi}{4 c_{1}} \right)=\PP\left(Z_{\mu} \geq \frac{2c_{1}}{\beta^{2}\xi}\right).
\end{aligned}
\end{equation*}
This implies the statement of the theorem.
\end{proof}

\begin{rem}
If $k,a$ and $b$ are constants, a more explicit lower bound for $\mathbb{P}%
(\tau <+\infty )$ is available without the assumption \eqref{PB}. Indeed,
starting with \eqref{tau*}, a straightforward calculation gives a lower
bound in terms of a gamma-distributed random variable $Z$ again, but this
time with parameter $\widehat{\mu }:=(\lambda _{0}k^{2}+a^{2})/(a^{2}\beta ).
$ More precisely,
\begin{equation*}
\mathbb{P}(\tau <\infty )\ \geqq \ \mathbb{P}(\tau ^{\ast }<\infty )\ = \
\mathbb{P}\left(Z_{\widehat{\mu }}\leqq \frac{2C}{a^{2}\beta }\langle
\varphi ,\phi _{0}\rangle_{D} ^{\beta }\right).
\end{equation*}
\end{rem}

\subsection{ A lower bound for the blowup time}

Our next goal is to obtain a lower bound for the blowup time $\tau.$ Since
the proofs of the following results are close to those in \cite{ALP} (where $%
b=0$), we omit them here.

\begin{theorem}
\label{THM7} Let the function $g$ be such that $g(0)=0$, $z\rightarrow g(z)/z
$ is increasing, and $g(z)\leq \Lambda z^{1+\beta}$ for some positive
constant $\Lambda.$ Then $\tau \geq \tau_*$, where
\begin{equation}
\tau _{\ast }=\inf \left\{t>0:\text{ }\int_{0}^{t}\exp (\beta
(N_{r}-A(r)))\left\| U^{D}(r,0)\varphi\right\| _{\infty }^{\beta }dr\geqq
\frac{1}{\Lambda \beta }\right\}.  \label{tau*2}
\end{equation}
Let us define for $0\leqq t<\tau _{\ast },$
\begin{equation*}
J(t)=\left( 1-\Lambda \beta \int_{0}^{t}\exp (\beta (N_{r}-A(r)))\left\|
U^{D}(r,0)\varphi \right\|_{\infty }^{\beta }dr\right) ^{-1/\beta }.
\end{equation*}
Then the solution $u$ of \eqref{2.1} satisfies, for $x\in D,$ $0\leqq t<\tau
_{\ast },$ $\mathbb{P}$-a.s.
\begin{equation}
0\leqq u(x,t)\leqq J(t)\exp (N_{t}-A(t))U^{D}(t,0)\varphi (x).  \label{ubu}
\end{equation}
\end{theorem}

\begin{rem}
More precisely, the proof of this theorem shows that the mild solution $v$
of \eqref{2.2} satisfies \eqref{ubu}\ without the factor $exp(N_{t}).$ By
Theorem \ref{THM8}, $v$ is also the weak solution of \eqref{2.2}, hence the
weak solution $u(\cdot ,t)=\exp (N_{t})v(\cdot ,t)$ of \eqref{2.1} satisfies %
\eqref{ubu}.
\end{rem}

\begin{coro}
\label{CC4} Assume that%
\begin{equation*}
\Lambda \beta \int_{0}^{\infty }\exp [\beta (N_{r}-A(r))]\left\|
U^{D}(r,0)\varphi \right\|_{\infty }^{\beta }dr<1.
\end{equation*}%
Then the solution $u$ of \eqref{2.1} satisfies \eqref{ubu} $\mathbb{P}$-a.s.
for all $t.$
\end{coro}

\begin{rem}
For the special choice of $\varphi =p\psi _{0},$ $p>0,$ the integrals
appearing \ in \eqref{tau*} and \eqref{tau*2} are the same exponential
functionals of $N.$ In fact, $U^{D}(r,0)\psi _{0}=\exp (-\lambda
_{0}K(r))\psi _{0},$ and $\tau _{\ast }$ becomes%
\begin{equation}
\tau _{\ast }=\inf \left\{t>0:\text{ }\int_{0}^{t}\exp \left[\:\:%
\phantom{1^l\!\!\!\!\!\!\!} \beta (N_{r}-\lambda _{0}K(r)-A(r))\right]%
\,dr\geqq \frac{p^{-\beta }}{\Lambda \beta } \left\| \psi _{0}\right\|
_{\infty }^{-\beta }\right\},  \label{taupsy2}
\end{equation}%
whereas
\begin{equation}
\tau ^{\ast }=\inf \left\{t>0:\int_{0}^{t}\exp \left[\:\:\phantom{1^l\!\!\!%
	\!\!\!\!}\beta (N_{r}-\lambda _{0}K(r)-A(r))\right]\,dr\geq \frac{p^{-\beta }%
}{C\beta } \langle \psi_{0} ,\phi _{0}\rangle_{D} ^{-\beta }\right\}.
\label{taupsy1}
\end{equation}
In fact $\tau _{\ast }\leqq \tau ^{\ast }$ if $C\leqq \Lambda ,$ since $%
\langle \psi_{0} ,\phi _{0}\rangle_{D} \leqq \| \psi _{0}\|_{\infty
}\int_{D}\phi _{0}(x)dx=\|\psi_{0}\|_{\infty }. $ In order to apply both
bounds simultaneously, we have to suppose $Cz^{1+\beta }\leqq g(z)\leqq
\Lambda z^{1+\beta },$ $z>0.$ It is therefore of interest to know the law of
the integral appearing in \eqref{taupsy2} and \eqref{taupsy1}$.$ This seems
possible only for $b^{H}=0, $ since, to our best knowledge, the law of
exponential functionals of fractional Brownian motion is still unknown. For
the moment it seems that only estimates of the type of those in Section \ref%
{section.3.1} are available. See also Theorem \ref{THM5} for $H>3/4.$
\end{rem}

\section{A sufficient condition for finite time blowup}

We consider now the mild form of \eqref{2.2} obtained in Proposition \ref%
{Proposition2}, and obtain a sufficient condition for finite time blowup.

\begin{theorem}
\label{THM6} Suppose that $g(z)\geq Cz^{1+\beta}$ and that there exists $%
w^*>0$ such that%
\begin{equation}
\exp (\beta A(w^*))\parallel U^{D}(w^*,0)\varphi \parallel _{\infty
}^{-\beta }<\beta C\int_{0}^{w^*}\exp (\beta N_{s})\,ds\text{ .\label{cond2}}
\end{equation}%
Then for the explosion time $\tau $ of~\eqref{2.1} there holds $\tau \leq
w^*.$
\end{theorem}

\begin{rem}
Inequality \eqref{cond2} is understood trajectorywise. Therefore $w^*$ is
random. \eqref{cond2} is harder to satisfy with a small initial condition $%
\varphi $ and with a small value of $C.$ Due to the different
interpretations of the integrals in $N$, the effects on blowup of $B$ and $%
B^{H}$ are different. If $N=0,$ \eqref{cond2} reads $\parallel
U^{D}(w^*,0)\varphi \parallel _{\infty }^{-\beta }<\beta Cw^*$ and in this
case $w^*$ {\color{black}
is deterministic; if in addition  $\varphi
=\psi _{0},$ \eqref{cond2} reads
}
 $\exp(\lambda _{0}\beta K(w^*))\parallel
\psi _{0}\parallel _{\infty }^{-\beta }<\beta Cw^*.$
\end{rem}

\begin{proof}We use the approach in \cite[Lemma 15.6]{Q-S}; see also   \cite{LDP}.  Suppose that $v(x,t),$ $x\in D,$ $t\geq 0,$ is a global
solution of \eqref{2.2}, and let $0<t<t^{\prime }.$ Using the semigroup
property of the evolution system $(U^{D}(t,r))_{0\leqq r<t}$ we obtain
\begin{eqnarray*}
	&&\exp \LL(\:\ph-A(t^{\prime },t)\RR)U^{D}(t^{\prime },t)v(\cdot ,t)(x) \\
	&=&\exp\LL(\:\ph-A(t^{\prime },t)\RR)U^{D}(t^{\prime },t)\left[ \exp
	\LL(\:\ph-A(t)\RR)U^{D}(t,0)\varphi (\cdot )\right] (x) \\
	&&+\exp \LL(\:\ph-A(t^{\prime },t)\RR)U^{D}(t^{\prime },t)\left[ \int_{0}^{t}\exp
	(-N_{r})\exp \LL(\:\ph-A(t,r)\RR)U^{D}(t,r)g(\exp (N_{r})v(\cdot ,r))(x)\,dr\right] (x) \\
	&=&\exp \LL(\:\ph-A(t^{\prime })\RR)U^{D}(t^{\prime },0)\varphi (\cdot )(x) \\
	&&+\int_{0}^{t}\exp (-N_{r})\exp \LL(\:\ph-A(t^{\prime },r)\RR)U^{D}(t^{\prime
	},r)g(\exp (N_{r})v(\cdot ,r))(x)\,dr \\
	&\geqq &\exp \LL(\:\ph-A(t^{\prime })\RR)U^{D}(t^{\prime },0)\varphi (\cdot )(x) \\
	&&+C\int_{0}^{t}\exp (\beta N_{r})\exp \LL(\:\ph-A(t^{\prime },r)\RR)U^{D}(t^{\prime
	},r)v(\cdot ,r)^{1+\beta }(x)\,dr.
\end{eqnarray*}
By Jensen's inequality
\begin{eqnarray*}
	U^{D}(t^{\prime },r)v(\cdot ,r)^{1+\beta }(x) &=&\int_{D}p^{D}(r,x;t^{\prime
	},y)v(y,r)^{1+\beta }\,dy \\
	&\geqq &\left( \int_{D}p^{D}(r,x;t^{\prime },y)v(y,r)\,dy\right) ^{1+\beta
	}=\left(\:\ph U^{D}(t^{\prime },r)v(\cdot ,r)(x)\right) ^{1+\beta }.
\end{eqnarray*}
Therefore%
\begin{equation*}
\exp \LL(\:\ph-A(t^{\prime },t)\RR)U^{D}(t^{\prime },t)v(\cdot ,t)(x)\geqq \exp
\LL(\:\ph-A(t^{\prime })\RR)U^{D}(t^{\prime },0)\varphi (x)
\end{equation*}
\begin{equation}
+C\int_{0}^{t}\exp (\beta N_{r})\left(\:\Ph \exp \LL(\:\ph-A(t^{\prime
},r)\RR)U^{D}(t^{\prime },r)v(\cdot ,r)(x)\right) ^{1+\beta }dr  \label{psy}.
\end{equation}
Let $\psi (t)$ be the last term in \eqref{psy}. Then, from the above inequality,
$$
\psi ^{\prime }(t) =C\exp (\beta N_{t})\left( \:\Ph\exp (-A(t^{\prime
},t))U^{D}(t^{\prime },t)v(\cdot ,t)(x)\right) ^{1+\beta }
\geqq C\exp (\beta N_{t})(\psi (t))^{1+\beta }
$$
Let now $\Psi (t):=$ $\int_{t}^{\infty }dz/z^{1+\beta }=\frac{1}{\beta }%
t^{-\beta },$ $t>0.$ Then%
\begin{equation*}
\frac{d}{dt}\Psi (\psi (t))=-\frac{\psi ^{\prime }(t)}{(\psi (t))^{1+\beta }}%
\leqq -C\exp (\beta N_{t}).
\end{equation*}
Hence%
$$
C\int_{0}^{t^{\prime }}\exp (\beta N_{s})\,ds \leqq \Psi (\psi (0))-\Psi
(\psi (t^{\prime }))
=\int_{\psi (0)}^{\psi (t^{\prime })}dz/z^{1+\beta }<\int_{\exp
	(-A(t^{\prime }))U^{D}(t^{\prime },0)\varphi (\cdot )(x)}^{\infty
}dz/z^{1+\beta }
$$
for all $x\in D$ and all $t^{\prime }>0.$ Therefore
$
\beta C\int_{0}^{t^{\prime }}\exp (\beta N_{s})\,ds\leqq \exp (\beta
A(t^{\prime }))\|U^{D}(t^{\prime },0)\varphi \|_{\infty
}^{-\beta }
$
for all $t^{\prime }>0.$ This contradicts \eqref{cond2}.
\end{proof}

{\color{black}
{\noindent\bf Acknowledgement} \ {\color{black}The authors are grateful to two anonymous referees for their valuable comments, which greatly improved our paper.}
The second- and third-named authors acknowledge the hospitality of Institut \'{E}lie Cartan de Lorraine, where part of this
work was done. The research of the second-named author was  partially supported by CONACyT (Mexico), Grant No. 652255. The fourth-named author would like to express her gratitude to the entire staff of the IECL for their hospitality and strong support during the completion of her Ph.D. dissertation there.
}


\begin{thebibliography}{99}

\bibitem{ALP} A. Alvarez, J.A. L\'{o}pez-Mimbela, N. Privault.
\newblock {
	Blowup estimates for a family of semilinear SPDEs with time-dipendent
	coefficients. } {\em Differential Equations and Applications} {\bf2} (2015),
201-219.

\bibitem{CyW} X. Chen, J. Wang. Intrinsic ultracontractivity for general
L\'evy processes on bounded open sets. \emph{Illinois J. Math.} \textbf{58}
(2014), 1117-1144.

\bibitem{Ch} P. Cheridito. \newblock {  Mixed fractional Brownian motion.}
\emph{Bernoulli} \textbf{7} (2001), 913-934.

\bibitem{Cho} P.L. Chow.
\newblock {Explosive solutions of stochastic
	reaction-diffusion equations in mean Lp-norm.} {\em J. Diff. Equations} {\bf 250}
(2011), 2567-2580.

\bibitem{DKL} M. Dozzi, E.T. Kolkovska, J.A. L\'{o}pez-Mimbela.
\newblock{Finite-time blowup and existence of global positive solutions
	of a semi-linear SPDE with fractional noise.} In: \textsl{Modern Stochastics
	and Applications,} V. Korolyuk, N. Limnios, Y. Mishura, L. Sakhno, G.
Shevchenko (Eds.), Springer 2014, 95-108.

\bibitem{D-K-LM-SAA} M. Dozzi, E.T. Kolkovska, J.A. L\'{o}pez-Mimbela.
Global and non-global solutions of a fractional reaction-diffusion equation
perturbed by a fractional noise. \emph{Stoch. Anal. Appl. } \textbf{38}
(2020), no. 6, 959-978.

\bibitem{DL} M. Dozzi, J.A. L\'{o}pez-Mimbela. Finite time blowup and
existence of global positive solutions of a semi-linear SPDE. \emph{%
	Stochastic Processes Appl.} \textbf{120} (2010), 767-776.

\bibitem{D} N.T. Dung. Tail estimates for exponential functionals and
applications to SDEs. \emph{Stochastic Processes Appl.} \textbf{128}, Issue
12, (2018), 4154-4170.

\bibitem{DII} N.T.\ Dung. The probability of finite-time blowup of a
semi-linear SPDE with fractional noise. \emph{Statist. Probab. Lett.}
\textbf{149} (2019), 86-92.

\bibitem{Fuj} H. Fujita. On some nonexistence and nonuniqueness theorems for
nonlinear parabolic equations, in \textsl{Nonlinear Functional Analysis},
Providence, R.I., 1970, \emph{Proc. Symp. Pure Math.} \textbf{18}(1) (1968)
105-113.

\bibitem{333} M.J. Garrido-Atienza, B. Maslowski, J. \v{S}nup\'arkov\'a.
Semilinear stochastic equations with bilinear fractional noise. \emph{%
	Discrete Contin. Dyn. Syst. Ser. B} \textbf{21} (2016), no. 9, 3075-3094.

\bibitem{Friedman} A. Friedman. \textsl{Partial differential equations of
	parabolic type}. Prentice-Hall 1964.

\bibitem{Klebaner} F.C. Klebaner. {\sl Introduction to stochastic calculus with applications. Second edition.} Imperial College Press, London, 2005.

\bibitem{KyS} P. Kim, R. Song, Intrinsic ultracontractivity of non-symmetric
L\'evy processes. \emph{Forum Math.} \textbf{21} (2009), 43-66.

\bibitem{LDP} M. Loayza, C.S. Da Paix\~{a}o.
\newblock { Existence and
	non-existence of global solutions for a semilinear heat equation on a
	general domain.} {\em Electron. J. Differential Equations} {\bf 2014}
(2014), No. 168, 1-9.

\bibitem{LP} J.A. L\'{o}pez-Mimbela, A. P\'{e}rez. Global and nonglobal
solutions of a system of nonautonomous semilinear equations with
ultracontractive L\'{e}vy generators. \emph{J. Math. Anal. Appl.} \textbf{423%
} (2015), 720-733.

\bibitem{LR} S.V. Lototsky, B.L. Rozovsky,
\newblock {\sl Stochastic partial
	differential equations.} Springer 2017.

\bibitem{M} Y. Mishura.
\newblock {\sl Stochastic calculus for fractional
	Brownian motion and related processes. } Springer Lecture Notes in
Mathematics 1929 2008.

\bibitem{MS} Y. Mishura, G. M. Shevchenko. Existence and uniqueness of the
solution of stochastic differential equation involving Wiener process and
fractional Brownian motion with Hurst index $H > 1/2$. \emph{Comm. Stat. -
	Theory and Methods} \textbf{40} (2011), 3492-3508.

\bibitem{MRS} Y. Mishura, K. Ralchenko, G. Shevchenko. Existence and
uniqueness of mild solution to stochastic heat equation with white and
fractional noises. \emph{Theory Probab. Math. Statist.} No. 98 (2019),
149-170.

\bibitem{N} D. Nualart.
\newblock {\sl The Malliavin calculus and related
	topics. } Springer Verlag 2006.

\bibitem{NR} D. Nualart, N. R\u{a}\c{s}canu.
\newblock {Differential
	equations driven by fractional Brownian motion. } {\em Collect. Math.} {\bf53}
(2002) 55-81.

\bibitem{NV} D. Nualart, P.-A. Vuillermot.
\newblock {Variational
	solutions for partial differential equations driven by a fractional noise.}
\emph{J. Funct. Anal.} \textbf{232} (2006), 390-454.

\bibitem{PZ} S. Peszat, J. Zabczyk.
\newblock {\sl Stochastic partial differential equations with
	L\'{e}vy noise. } Cambridge University Press 2007.

\bibitem{Q-S} P. Quittner; P. Souplet. \textsl{Superlinear parabolic
	problems. Blow-up, global existence and steady states.} Birkh\"{a}user
Verlag, Basel, 2007.

\bibitem{RalSche} K. Ralchenko, G. Shevchenko,
\newblock {Existence and uniqueness of mild solutions to fractional stochastic heat equation.} {\em Mod. Stoch. Theory Appl.} {\textbf{6}} (2019)
57-79.


\bibitem{Y} M. Yor.
\newblock {\sl Exponential functionals of Brownian
	motion and related processes. } Springer Verlag 2001.

\bibitem{Z} M. Z\"{a}hle,
\newblock {Integration with respect to fractional functions and
	stochastic calculus I. } {\em Prob. Theory Rel. Fields} {\bf111} (1998)
333-374.

\end{thebibliography}
\end{document}